\documentclass[a4paper]{amsart}

\DeclareMathOperator{\mean}{\mathbb{E}}

\DeclareMathOperator{\Cov}{\mathbb{C}ov}
\DeclareMathOperator{\Cor}{\mathbb{C}or}

\DeclareMathOperator{\meanappr}{E}

\DeclareMathOperator{\Id}{Id}
\DeclareMathOperator{\LandauO}{\mathcal{O}}

\renewcommand{\div}{\operatorname{div}}

\DeclareMathOperator{\supp}{supp}

\newcommand{\dd}{\mathrm{d}}
\newcommand{\dif}{\, \dd}

\DeclareMathOperator{\Fre}{{D}}

\DeclareMathOperator*{\argmax}{arg~max}
\DeclareMathOperator*{\argmin}{arg~min}

\newcommand{\nobs}{{K}}

\newcommand{\appr}{\text{approx}}

\newcommand{\MAP}{\mathrm{MAP}}

\newcommand{\fixed}{\star}

\newcommand{\nom}{\Theta} %
\newcommand{\datamark}{\delta}
\newcommand{\data}{\eta^\datamark}
\newcommand{\domain}{\mathcal{D}}
\newcommand{\boundary}{\partial \domain}

\newcommand{\prior}[1]{\pi_{#1}}
\newcommand{\posterior}[1]{{\prior{#1}^\datamark}}

\usepackage[utf8]{inputenc} %
\usepackage[english]{babel} %
\usepackage{bbm} %
\usepackage[tbtags]{mathtools}
\usepackage{amssymb,amsmath,amsthm} %
\usepackage{thmtools}
\usepackage{color} %
\usepackage{graphicx} %
\usepackage[hidelinks]{hyperref} %
\usepackage{cleveref} %
\usepackage{hyperref}
\usepackage{booktabs} %
\usepackage{enumerate} %
\usepackage{lipsum} %
\usepackage{lmodern} %
\usepackage{dsfont} %
\usepackage{tikz} %
\usepackage[export]{adjustbox}

\usepackage{pgfplots} 
\pgfplotsset{compat=newest}
\pgfplotsset{plot coordinates/math parser=false}
\newlength\figureheight
\newlength\figurewidth 

\theoremstyle{plain}
\newtheorem{theorem}{Theorem}[section]
\newtheorem{lemma}[theorem]{Lemma}
\newtheorem{corollary}[theorem]{Corollary}

\theoremstyle{definition}

\theoremstyle{remark}
\newtheorem{remark}[theorem]{Remark}

\graphicspath{ {./fig/} }

\usepackage[foot]{amsaddr}
\title{Local sensitivity analysis for Bayesian inverse problems}
\author{Jürgen Dölz}
\author{David Ebert}
\address{Institute for Numerical Simulation, University of Bonn, Friedrich--Hirzebruch--Allee~7, 53115~Bonn, Germany}
\email{doelz@ins.uni-bonn.de, ebert@ins.uni-bonn.de}
\thanks{The authors gratefully acknowledge the access to the Marvin cluster of the University of Bonn. The work of the authors was partially supported by the Deutsche Forschungsgemeinschaft (DFG, German Research Foundation), project 501419255. The authors also received support from the DFG under Germany's Excellence Strategy, project 390685813.}

\date{\today}

\begin{document} 

\begin{abstract}
	We present an extension of local sensitivity analysis, also referred to as the perturbation approach for uncertainty quantification, to Bayesian inverse problems.
    More precisely, we show how moments of random variables with respect to the posterior distribution can be approximated efficiently by asymptotic expansions. This is under the assumption that the measurement operators and prediction functions are sufficiently smooth and their corresponding stochastic moments with respect to the prior distribution exist.
    Numerical experiments are presented to the illustrate the theoretical results.
\end{abstract}

\maketitle

\section{Introduction}

\subsection{Motivation}

One of the key goals in Bayesian inverse problems is to update prior beliefs on the distribution of random variables by additional knowledge. These random variables often correspond to physical phenomena, where the additional knowledge is given by measurements \cite{KS2005}.
The updated information is then given as a posterior distribution, which can be used to compute statistical quantities of interest such as stochastic moments.
Let $(\Omega,\mathcal{A},\prior{})$ be a probability space, $X$ a Banach space, and $\xi:\Omega\to X$ a random variable. 
Given an observable mapping $Q:X \to \mathbb{R}^\nobs$, a prior distribution $\prior{\xi} = \prior{} \circ \xi^{-1}$, and data $\data$ observed under additive Gaussian noise with covariance matrix $\Sigma$, independent of $\xi$, the posterior distribution~$\posterior{\xi}$ is given through the Radon--Nikodým derivative
\begin{align}
    \frac{\dd \posterior{\xi}}{\dd \prior{\xi} }(x) 
    \propto \exp\big( - \tfrac{1}{2} \| \data - Q(x) \|^2_\Sigma \big) , \label{eq:posterior1}
\end{align}
cf. \cite{Stu2010}. Stochastic moments of a quantity of interest $R$ defined on $X$ according to the posterior distribution, the \emph{posterior moments}, are given by
\begin{align}
    \int \limits_X R(x) \otimes \cdots \otimes R(x) ~\dd\posterior{\xi}(x). \label{eq:integral}
\end{align}
Approximating the latter computationally is a nontrivial task, especially since $Q$ and $R$ are often the solution operators of partial differential equations and thus computationally expensive when implemented. Nevertheless, approximating \eqref{eq:integral} numerically has been subject to intensive research efforts over the past years.

\subsection{Numerical methods for Bayesian inverse problems}
The common approach to compute stochastic moments \eqref{eq:integral} computationally are sampling based methods. 
Among these, \emph{Markov Chain Monte Carlo} (MCMC) \emph{methods} \cite{DKPR1987,Has1970} aim at drawing samples from the posterior distribution to estimate \eqref{eq:integral} by a Monte Carlo estimator.
Like all Monte Carlo estimators they provide a relatively slow convergence rate of $\mathcal{O}(N^{-1/2})$ with respect to number of samples $N$.
While Monte Carlo estimators are naturally robust to the dimension $d$ of the parameter space $X$, finding a proposal kernel for MCMC estimators that is robust as $d\to \infty$ is often a nontrivial task, cf. \cite{CRSW2013,RR2001,RS2018,Tie1998}.
Over the years MCMC methods have seen various proposals for designing efficient proposal kernels for the Markov chain. Examples are multilevel MCMC \cite{DKST2019,HSS2013,LDF+2023}, adaptive MCMC \cite{HST2001,Ros2011}, Hamiltonian Monte Carlo \cite{DKPR1987,GC2011,HG2014}, and hybrid Monte Carlo \cite{BPSS2011,DKPR1987}, which all aim to improve the way and at which computational cost proposals for the chain are generated.

MCMC methods to compute stochastic moments can be considered as quadrature methods and, thus, a natural approach to overcome the MCMC's slow convergence rate is to employ improved quadrature formulas. To do so \eqref{eq:posterior1} is commonly substituted into \eqref{eq:integral} and a quadrature formula for the prior measure is employed. 
In comparison to MCMC methods, this approach requires the normalization constant of the posterior measure to be numerically calculated explicitly. 
Exploiting this perspective and additionally available smoothness of the integrand, \emph{sparse grid methods} \cite{MZ2009,SS2013}, \emph{polynomial chaos} (PC) \emph{methods} 
\cite{MN2009,MNR2007,RLPM2012}, and (higher-order) \emph{quasi-Monte Carlo} (QMC) \emph{methods} \cite{DGLS2017,DKOS2025,GP2018,SST2017} can be employed to improve the convergence rate. 
These methods can be further accelerated when they are combined to a \emph{multilevel quadrature} algorithm~\cite{DGLS2017,DHJM2022,SST2017}. 

As an alternative to quadrature algorithms, \emph{Kalman Filters} \cite{Kal1960} and their extensions such as \emph{ensemble Kalman Filters} \cite{Eve1994} and \emph{polynomial chaos Kalman Filters} \cite{RLPM2012} are methods which are designed to provide a quick estimate of the mean and variance of certain parameters.
Computational speed is favored in these methods and, except for linear systems and Gaussian noise, the approximation properties of Kalman Filters are often not well understood.
However, there is some progress being made in the context of Bayesian inverse problems \cite{ESS2015,ILS2013,LTT2016,RLPM2012,SS2017}.

Besides the above mentioned sampling based methods, an alternative approach for the numerical evaluation of \eqref{eq:integral} is to first approximate the posterior density by analytical methods and then to estimate it by means of numerical methods. 
Significant gains can be made if the analytical surrogate can be more efficiently approximated than the original expression \eqref{eq:integral}. 
The most prominent example is the \emph{Laplace approximation} \cite{TK1986}, which approximates the posterior distribution by a Gaussian distribution centered at the maximum a posteriori (MAP) estimate. The covariance of this Gaussian distribution is based on a second-order Taylor expansion of the logarithm of the posterior density around the MAP estimate. 
It can be shown that this approximation converges to the true posterior when the posterior distribution is unimodal and the \emph{noise level} decreases \cite{HK2022,SSW2020}. 
Otherwise, when the approximation of the posterior distribution by a Gaussian distribution is an insufficient approximation, the Laplace approximation can be used to improve quadrature algorithms such as Monte Carlo, QMC methods, or sparse grids \cite{SS2016,SSW2020}. 
In this situation, there hold similar considerations as for sampling based methods.

\subsection{Local sensitivity analysis for forward problems}
For forward problems in uncertainty quantification, i.e., computing statistical quantities of interest of some quantity with random input data, \emph{local sensitivity analysis} is an attractive alternative to sampling based methods \cite{Smi2013}. In the context of partial differential equations local sensitivity analysis is also referred to as the \emph{perturbation approach}, see, e.g.,  \cite{BN2014,CS2013,Dol2020,DE2024,EJ2020,HPS2013,HSS2008a}. Local sensitivity analysis assumes that the input data are modeled as random fluctuations around a
reference state $x_0\in X$, i.e., as
\begin{align}\label{eq:intropert}
x(\omega)
=
x_0+\xi(\omega)\in X,
\end{align}
and expand the uncertainty-to-solution map into a truncated Taylor expansion around $x_0$. Computing the corresponding derivatives is often not more expensive than the model evaluation itself. 
As the Taylor expansion admits an asymptotic truncation error with respect to amplitude of the uncertainty $\xi(\omega)$, local sensitivity analysis is most suited to efficiently calculate stochastic quantities of interest in applications with small scale uncertainty to high precision. In this regime it often outperforms sampling based algorithms. 
As it turns out, this is highly relevant in engineering applications \cite{BR1988,DESZ2024,MI1999,Ibr1987,ID1976,SAJF2018}. 
Compared to many sampling bases approaches, local sensitivity analysis can be considered to be immune to the curse of dimensionality.

To the best of our knowledge, the local sensitivity analysis has only been considered for forward uncertainty quantification. An analogue for Bayesian inverse problems is missing so far.

\subsection{Contributions}
In this article, we extend and analyze the framework of local sensitivity analysis for forward problems to Bayesian inverse problems. Assuming random input data of the form \eqref{eq:intropert}, a prior distribution for the perturbations, and Gaussian measurement noise, we provide approximations of posterior moments in terms of Taylor expansions. 
More precisely, the contributions of this paper are as follows.
\begin{enumerate}
    \item Given asymptotic expansions for $Q$ and $R$, we provide and analyze asymptotic expansions of \eqref{eq:integral}, i.e., we derive a local sensitivity analysis with respect to the posterior measure.
    \item As local sensitivity analysis is sensitive to the reference point, we discuss how the reference point can be iteratively improved and how this connects to classically regularized inverse problems for the specific case of the posterior mean of the parameter.
    \item We discuss the numerical implementation and provide extensive numerical examples for the Darcy flow with log-normal permeability and small noise level as well as the Lotka--Volterra ODE with Brownian bridge perturbation, illustrating the validity of our numerical examples.
\end{enumerate}
We note that a similar concept based on a perturbation analysis of the posterior measure with respect to the prior measure, i.e., how the posterior is affected by changes in the prior, is known in the context \emph{robust Bayesian analysis} \cite{Ber1990,BMP+1994,IR2000}. This is different from the local sensitivity analysis we consider here, as we consider the prior as fixed and given and our investigations focus on perturbations of $Q$ and $R$ and how these perturbations affect \eqref{eq:integral}.

\subsection{Outline}
The rest of this article is organized as follows. In the following \cref{sec:setting}, we recall the basics of Bayesian inverse problems, statistical moments, and Taylor expansions in Banach spaces. Then, in \cref{sec:perturbation}, we introduce and analyze local sensitivity analysis for posterior moments in Bayesian inverse problems. In \cref{sec:impr} we discuss how the reference point of the expansions for the posterior mean can be improved and how it relates to classically regularized inverse problems. \Cref{sec:implementation} is concerned with remarks on implementation when using affinely parametrized random variables. Finally, in \cref{sec:examples}, we illustrate our theoretical findings with numerical examples, before we draw our conclusions in \cref{sec:conclusion}.

\section{Preliminaries} \label{sec:setting}

\subsection{Bayesian inverse problems}
In the following we consider a \emph{forward response map} $G\colon X\to Y$ between two Banach spaces from which we refer to $X$ as the \emph{parameter space} and $Y$ as the \emph{observable space}. 
With $Y$ being possibly infinite dimensional, we assume to have a bounded and linear \emph{observation operator} $O\in\mathcal{L}(Y;\mathbb{R}^\nobs)$ to our disposal which allows us to perform $K$ ``measurements'' on the observation space. 
We refer to the concatenation $Q=O\circ G\colon X\to\mathbb{R}^\nobs$ as the \emph{measurement operator} which maps parameters from $X$ to $\nobs$ ``measurements''. 
As is common for Bayesian inverse problems we assume that the measurements are subject to additive Gaussian noise $\varepsilon\sim\mathcal{N}(0,\Sigma)$, with $\Sigma\in\mathbb{R}^{\nobs\times \nobs}$ being a symmetric, positive definite covariance matrix, and that the parameter in $X$ is an $X$-valued random variable $\xi$ defined on some probability space $(\Omega,\mathcal{A},\prior{})$. 
We call $\prior{\xi} = \prior{} \circ \xi^{-1}$ the \emph{prior} distribution of $\xi$ and assume that the noise $\varepsilon$ is independent of the model parameter $\xi$.
Given this model, the ``physically'' accessible measurements are given as realizations $\data\in\mathbb{R}^\nobs$ of the random variable
\begin{align} \label{eq:noisemodel}
	\eta = Q(\xi) + \varepsilon , 
	\qquad (\xi,\varepsilon) \sim \prior{\xi} \otimes \mathcal{N}(0,\Sigma).
\end{align}
Given a realization $\data\in\mathbb{R}^\nobs$ of this random variable we define the \emph{(unnormalized) likelihood} $\nom\colon X\to\mathbb{R}$ as
\[
\nom(x) 
= \exp( -\Phi(x) ) ,
\]
and we refer to 
\begin{align}
	\Phi(x) = \frac{1}{2} \big\|\data - Q(x)\big\|^2_\Sigma \label{eq:potential}
\end{align} 
as the \emph{potential} of the likelihood with $\langle \cdot , \cdot \rangle_\Sigma = \langle \Sigma^{-1} \cdot , \cdot \rangle_{\mathbb{R}^\nobs}$ being the scalar product that induces the norm $\|\cdot \|^2_\Sigma = \langle \cdot,\cdot \rangle_\Sigma$.
The following theorem allows to condition the prior distribution $\prior{\xi}$ to the measurements $\data$.
We call distributions conditioned on the measurements $\data$ \emph{posterior} distributions.
\begin{theorem}[{\cite[Theorems 6.29 and 6.31]{Stu2010}}]
Assume that the potential \eqref{eq:potential} is $\prior{\xi}$-measur\-able. Then the posterior distribution of $\xi$ conditioned to $\data$ exists and is denoted by $\posterior{\xi} = \prior{\xi|\eta=\data}$. It is absolutely continuous with respect to $\prior{\xi}$ and given through the Radon--Nikodým derivative
\begin{align}
	\frac{\dd \posterior{\xi}}{\dd \prior{\xi}} (x) 
	=\frac{\nom(x)}{\int_X \nom(x)\,\dd\prior{\xi}(x)}  . \label{eq:posterior}
\end{align}
\end{theorem}

\subsection{Stochastic moments}
Let $R\colon X\to Z$ be a \emph{prediction function} taking values in some Hilbert space $Z$ and mapping parameters to some \emph{quantity of interest}. 
An illustration of the prediction function and the quantities of the Bayesian setting is given in \cref{fig:sketch}.
\begin{figure} 
	\begin{center}
		\begin{tikzpicture}[scale=1.5]
			\node[] (O) at (-1,1) {$\Omega$};
			\node[] (X) at (0,1) {$X$};
			\node[] (Y) at (1.5,2) {$Y$};
			\node[] (R) at (3,2) {$\mathbb{R}^\nobs$};
			\node[] (n) at (4.5,2) {$\eta \in \mathbb{R}^\nobs$};
			\node[] (Z) at (1.5,.5) {$Z$};
			\draw[->] (O)--(X) node [above,midway]{$\xi$};
			\draw[->] (X)--(Y) node [above,midway]{$G$};
			\draw[->] (X)--(Z) node [above,midway]{$R$};
			\draw[->] (Y)--(R) node [above,midway]{${O}$};
			\draw[->] (R)--(n) node [above,midway]{$+\varepsilon$};
			\draw[->] (X)--(R) node [below,near end]{$Q$};
		\end{tikzpicture}
		\caption{Illustration of the mapping properties considered for computing stochastic moments with respect to the posterior distribution.} \label{fig:sketch}
	\end{center}
\end{figure}
For $R\in L^1_{\prior{\xi}}(X;Z)$ the \emph{mean} or \emph{expected value} in terms of the prior distribution is defined as
\begin{subequations} \label{eq:voi}
\begin{align}
	\mean[R] &= \int_X R(x) \,\dd\prior{\xi}(x) , \label{eq:voi:expprior} \\
\intertext{and the posterior mean is the mean conditioned on realization $\data$  of $\eta$}
	\mean_\posterior{}[ R ]
	&= \mean [ R \,|\, \eta = \data ] 
	= \mean \bigg[ R \frac{\nom}{\mean [\nom]} \bigg] \label{eq:voi:exppost} . \\
\intertext{For $R_1,R_2\in L^2_{\prior{\xi}}(X;Z)$, and with $\otimes$ denoting the Hilbertian tensor product, we obtain the correlations}
    \Cor[R_1,R_2] &= \mean[R_1 \otimes R_2] , \\
    \Cor_{\posterior{}}[R_1,R_2] &= \mean_{\posterior{}}[R_1 \otimes R_2], \\
\intertext{and covariances}
	\Cov[R_1,R_2]
	&= \mean\big[ \big(R_1 - \mean[R_1] \big) \otimes \big(R_2 - \mean[R_2] \big) \big] , \label{eq:voi:cov} \\
    \Cov_\posterior{}[R_1,R_2]
	&= \mean_\posterior{}\big[ \big(R_1 - \mean_\posterior{}[R_1] \big) \otimes \big(R_2 - \mean_\posterior{}[R_2] \big) \big] . \label{eq:voi:cov_post}
\end{align}
\end{subequations}
Additionally, for $R=R_1=R_2$, we introduce the shorthand notation
\begin{align*}
    \Cov[R] &= \Cov[R,R] , \\
    \Cov_\posterior{}[R] &= \Cov_\posterior{}[R,R] , \\
    \Cor[R] &= \Cor[R,R] , \\ 
    \Cor_\posterior{}[R] &= \Cor_\posterior{}[R,R] .
\end{align*}
For the rest of this article, if expected values or covariances are not indexed, they are to be understood in the sense of the prior distribution.

\begin{remark}
We assume $Z$ to be a Hilbert space to avoid technicalities using the tensor product occurring in the definitions of correlation and covariance. We refer to \cite{JK2015} for a discussion concerning the details when $Z$ is a Banach space.
\end{remark}

\subsection{Taylor expansions in Banach spaces}
For the following results, we assume that the measurement operator and the variable of interest are at least three times continuously differentiable on an open convex neighborhood $U \subset X$ of a reference point~$x_0$, i.e.
\begin{align*}
	Q \in C^3(U;\mathbb{R}^\nobs) , && R \in C^3(U;Z) .
\end{align*}
For a Banach space $W$, an open, convex subset $U\subset X$ and $f\in C^{k+1}(U;W)$ we denote the $k$-th Fr\'echet derivative at $x_0\in U$ evaluated in the direction $[h_1,\ldots,h_k]\in X^k$ by $\Fre_{x_0}^kf[h_1,\ldots,h_k]$. For $[h_1,\ldots,h_k]=[h,\ldots,h]$ we abbreviate $\Fre_{x_0}^kf[h]$ and we write $\Fre_{x_0}f[h]=\Fre_{x_0}^1f[h]$. We recall that $\Fre_{x_0}^kf\in \mathcal{L}^{(k)}(U;W)$, i.e., $\Fre_{x_0}^kf$ is a bounded $k$-linear mapping on $X^k$. The Taylor expansion of $f$ at $x_0\in U$ in direction $h\in X$ with $x_0+h\in U$ is given by
\[
f(x_0+h)
=
f(x_0)
+
\Fre_{x_0}f[h]
+
\ldots
+
\frac{1}{k!}\Fre_{x_0}^kf[h]
+
R_k(x_0,h)
\]
with remainder
\[
R_k(x_0,h)
=
\int_0^1
\frac{(1-t)^k}{k!}\Fre_{x_0+th}^{k+1}f[h]
\dd t.
\]
The assumption $f\in C^{k+1}(U;W)$ implies
\[
\|\Fre_{x_0+th}^{k+1}f\|_{\mathcal{L}^{(k+1)}(U;\mathbb{R})}
\leq
\|f\|_{C^{k+1}(U;W)}
<
\infty
\]
for all $t\in [0,1]$, implying
\[
\|R(x_0,h)\|_W\leq\frac{\|f\|_{C^{k+1}(U;W)}}{(k+1)!}\|h\|_X^{k+1}
\]
and thus $R(x_0,h)\in\LandauO(\|h\|_X^{k+1})$ for all $x_0\in U$, $h\in X$ with $x_0+h\in U$, cf., e.g., \cite[Chapter 4.6]{Zei1986}.

The following lemma is an immediate consequence and we state it for later reference but without proof.
\begin{lemma}\label{lem:doubletaylor}
Let $X,W$ Banach spaces, $U\subset X$ an open convex subset, and $f\in C^{k+1}(U;W)$. 
Let $x_0\in U$, $h,g\in X$ and $x_0+h,x_0+g\in U$. 
Then there holds
\[
f(x_0+h)
=
f(x_0+g)
+
\Fre_{x_0}f[h-g]
+
\ldots
+
\frac{1}{k!}\Fre_{x_0+g}^kf[h-g]
+
R_k(x_0+g,h-g)
\]
with
\begin{align*}
\|R_k(x_0+g,h-g)\|_W
&\leq
\frac{\|f\|_{C^{k+1}(U;W)}}{(k+1)!}\|h-g\|_X^{k+1}\\
&\leq
\frac{C_k\|f\|_{C^{k+1}(U;W)}}{(k+1)!}\max\{\|h\|_X,\|g\|_X\}^{k+1}
\end{align*}
and thus
\[
R_k(x_0+g,h-g)
=
\LandauO\Big(\max\{\|h\|_X,\|g\|_X\}^{k+1}\Big),
\]
with the constant in the $\LandauO$-notation only depending on $f$ and $k$.
\end{lemma}

\section{Local sensitivity analysis for Bayesian inverse problems} \label{sec:perturbation}

\subsection{Perturbation model}
The common assumption for local sensitivity analysis is that the uncertainty of the data in the parameter space can be modeled as fluctuations around a reference value $x_0\in X$, i.e., as
\begin{equation}\label{eq:domainmap}
x_{\xi}(\omega)
=
x_0+\xi(\omega).
\end{equation}
\begin{remark}
We note that local sensitivity analysis often scales the perturbation $\xi$ by an additional parameter, i.e., $\xi(\omega)=\alpha\tilde{\xi}(\omega)$, yielding convergence rates in $\alpha$. We will see that for the Bayesian setting the unscaled version is preferable, yielding convergence rates in $\|\xi\|_X$.
\end{remark}
Assuming $Q\in C^4(X;\mathbb{R}^\nobs)$ and $R\in C^4(X;Z)$, the mappings of our Bayesian setting can be expressed as
\begin{subequations} \label{eq:series}
	\begin{align}
		Q_{x_{\xi}}
		=
		Q(x_{\xi})
		&= Q_{x_0} 
		+ \Fre_{x_0}Q[\xi]
		+ \frac{1}{2} \Fre_{x_0}^2Q[\xi]+ \frac{1}{6} \Fre_{x_0}^3Q[\xi]
		+ \LandauO(\|\xi\|_X^4), \\
		R_{x_{\xi}}
		=
		R(x_{\xi})
		&= R_{x_0} 
		+ \Fre_{x_0}R[\xi]
		+ \frac{1}{2} \Fre_{x_0}^2R[\xi]
		+ \frac{1}{6} \Fre_{x_0}^3R[\xi]
		+ \LandauO(\|\xi\|_X^4) . 
		\intertext{Accordingly, the likelihood $\nom \in C^4(X;\mathbb{R})$ can be expressed as}
		\nom_{x_{\xi}}
		=
		\nom(x_{\xi})
		&= \nom_{x_0} 
		+ \Fre_{x_0}\nom[\xi]
		+ \frac{1}{2} \Fre_{x_0}^2\nom[\xi]
		+ \frac{1}{6} \Fre_{x_0}^3\nom[\xi]
		+ \LandauO(\|\xi\|_X^4),
	\end{align}
\end{subequations}
where
\begin{subequations} \label{eq:nom}
	\begin{align} 
		\nom_{x_0} 
		&= \exp\Big(- \frac{1}{2} \|\data-Q_{x_0}\|^2_\Sigma\Big) , \label{eq:nom_0} \\
		\Fre_{x_0}\nom[\xi] 
		&= \nom_{x_0} ~\langle \data-Q_{x_0} , \Fre_{x_0}Q[\xi] \rangle_\Sigma , \label{eq:nom_d_0}\\
		\Fre_{x_0}^2\nom[\xi] 
		&= \nom_{x_0} \label{eq:nom_d2_0} 
		\Big( 
		\langle \data-Q_{x_0} , \Fre_{x_0}Q[\xi] \rangle_\Sigma^2 
		- \| \Fre_{x_0}Q[\xi] \|_\Sigma^2
		+ \langle \data-Q_{x_0}, \Fre_{x_0}^2Q[\xi] \rangle_\Sigma 
		\Big) , \\
		\Fre_{x_0}^3\nom[\xi] 
		&= \label{eq:nom_d3_0}
		\Fre_{x_0}^2\nom[\xi]~\langle \data-Q_{x_0}, \Fre_{x_0}Q[\xi]\rangle_\Sigma \\
		&\quad+ \nom_{x_0} \Big( 2 \langle \data-Q_{x_0}, \Fre_{x_0}Q[\xi]\rangle_\Sigma \Big( \langle \data - Q_{x_0}, \Fre_{x_0}^2Q[\xi]\rangle_\Sigma 
		- \| \Fre_{x_0}Q[\xi] \|^2_\Sigma \Big) \nonumber \\
		&\qquad\qquad -3 \langle \Fre_{x_0}Q[\xi], \Fre_{x_0}^2Q[\xi] \rangle_\Sigma
		 + \langle \data-Q_{x_0}, \Fre_{x_0}^3Q[\xi] \rangle_\Sigma \Big) \nonumber .
	\end{align}
\end{subequations}
Note that only the derivatives \cref{eq:nom_d_0,eq:nom_d2_0,eq:nom_d3_0} depend on the random variable $\xi$.
To simplify presentation we drop the argument $[\xi]$ from our notation.

\subsection{Expansions for first and second moments}
Inserting the series expansions \eqref{eq:series} into the posterior moments \eqref{eq:voi:exppost} and \eqref{eq:voi:cov_post}, we can formulate series approximations for these quantities of interest. In the following we do this up to second order, noting that there is no obstruction other than additional regularity requirements to go even higher order.
\begin{theorem} \label{th:approximations3}
	Let $\xi \in L_{\prior{}}^3(\Omega;X)$, $Q \in C^3(X;\mathbb{R}^\nobs)$, and $R \in C^3(X;Z)\cap L^2_{\prior{\xi}}(X;Z)$. Then there holds
\begin{subequations} \label{eq:approx3}
\begin{align}
	\mean_\posterior{}[ R_{x_{\xi}} ]
	&= R_{x_0} 
    + \mean[\Fre_{x_0}R] \label{eq:approx3:exp_R} \\ 
	&\quad + \frac{1}{2} \bigg( \mean[\Fre_{x_0}^2R] 
    + 2 \Cov[\Fre_{x_0}R, \langle\data-Q_{x_0}, \Fre_{x_0}Q\rangle_\Sigma ] \nonumber \bigg) \nonumber \\ 
    &\quad + \LandauO(\|\xi\|_{L_{\prior{}}^3(\Omega;X)}^3) \nonumber , \displaybreak[0] \\
	\Cor_\posterior{}[ R_{x_{\xi}} ]
	&= R_{x_0} \otimes R_{x_0} 
    + \bigg(\mean[\Fre_{x_0}R] \otimes R_{x_0} + R_{x_0} \otimes \mean[\Fre_{x_0}R]\bigg) \label{eq:approx3:exp_RxR} \\
	& \quad + \frac{1}{2} \bigg( \mean[\Fre_{x_0}^2R] \otimes R_{x_0} 
    + R_{x_0} \otimes \mean[\Fre_{x_0}^2R] 
	+ 2 \Cor[\Fre_{x_0}R] \nonumber \\
    & \qquad + 2 \Cov[\Fre_{x_0}R \otimes R_{x_0} + R_{x_0} \otimes \Fre_{x_0}R,\langle\data-Q_{x_0}, \Fre_{x_0}Q\rangle_\Sigma] \bigg) \nonumber \\ 
	&\quad + \LandauO(\|\xi\|_{L_{\prior{}}^3(\Omega;X)}^3) \nonumber, \displaybreak[0] \\
    \Cov_\posterior{} [R_{x_{\xi}}] 
	&= \Cov[\Fre_{x_0}R] 
    + \LandauO(\|\xi\|_{L_{\prior{}}^3(\Omega;X)}^3) \label{eq:approx3:cov_R}.
\end{align}
\end{subequations}
\end{theorem}
\begin{proof}
    To show the assertion for the mean \eqref{eq:approx3:exp_R}, we need to calculate the derivatives of
\begin{align*}
	R_{x_{\xi}} \frac{\nom_{x_{\xi}}}{\mean[\nom_{x_{\xi}}]}\in C^3(X;\mathbb{R})
\end{align*}
with respect to $\xi$ and then take the mean. This leads to
\begin{align*}
	\Fre_{x_0} \bigg( R \frac{\nom}{\mean[\nom]} \bigg)
	&= (\Fre_{x_0}R) \frac{\nom_{x_0}}{\mean[\nom_{x_0}]} 
	+ R_{x_0} \frac{\Fre_{x_0}\nom}{\mean[\nom_{x_0}]} 
	- R_{x_0} \frac{\nom_{x_0}(\Fre_{x_0}\mean[\nom])}{\mean[\nom_{x_0}]^2} , \displaybreak[0] \\
	\Fre_{x_0}^2\bigg( R \frac{\nom}{\mean[\nom]} \bigg)
	&= (\Fre_{x_0}^2R) \frac{\nom_{x_0}}{\mean[\nom_{x_0}]}
	+ 2 (\Fre_{x_0}R) \frac{\Fre_{x_0}\nom}{\mean[\nom_{x_0}]} 
	- 2 (\Fre_{x_0}R) \frac{\nom_{x_0} (\Fre_{x_0} \mean[\nom])}{\mean[\nom_{x_0}]^2} \\
	& \qquad + R_{x_0} \frac{\Fre_{x_0}^2\nom}{\mean[\nom_{x_0}]} 
	- 2 R_{x_0} \frac{(\Fre_{x_0}\nom) (\Fre_{x_0}\mean[\nom])}{\mean[\nom_{x_0}]^2} \\  
    & \qquad - R_{x_0} \frac{\nom_{x_0} (\Fre_{x_0}^2\mean[\nom])}{\mean[\nom_{x_0}]^2} 
	+ 2 R_{x_0} \frac{\nom_{x_0} (\Fre_{x_0}\mean[\nom])^2}{\mean[\nom_{x_0}]^3} . 
\end{align*}
Now, since the $k$-th Fr\'echet derivative acts as a bounded $k$-linear operator in its directions of evaluation there holds
\begin{align*}
    \Fre_{x_0}^k\mean[\nom]
    &= \Fre_{x_0}^k \int \limits_X \nom \dif \prior{\xi} 
    = \int \limits_X \Fre_{x_0}^k\nom \dif \prior{\xi}
    = \mean[\Fre_{x_0}^k\nom]
    \qquad
    k=1,2,
\end{align*}
and $\xi \in L^2_{\prior{}}(\Omega)$ implies
\begin{align*}
	\big|\mean[\Fre_{x_0}^k\nom]\big|
	\leq
	\mean\big[|\Fre_{x_0}^k\nom|\big]
	\leq
	\|\Fre_{x_0}^k\nom\|_{\mathcal{L}^{(k)}(X;\mathbb{R})}\|\xi\|_{L_{\prior{}}^k(\Omega;X)}
	<\infty,
	\qquad
	k=1,2,
\end{align*}
i.e., these quantities are well defined. Further, \eqref{eq:nom_0} implies $\mean[\nom_{x_0}] = \nom_{x_0}$, such that applying the mean, exploiting $\Fre_{x_0}\mean[\nom]=\mean[\Fre_{x_0}\nom]$, and \eqref{eq:nom_d_0}, we get
\begin{align*}
	\mean_\posterior{}[ \Fre_{x_0} R ] 
	&= \mean \bigg[ (\Fre_{x_0} R ) \frac{\nom_{x_0}}{\mathbb{E}[\nom_{x_0}]} \bigg]
	+ \mean \bigg[ R_{x_0} \frac{\Fre_{x_0} \nom}{\mathbb{E}[\nom_{x_0}]} \bigg] 
	- \mean \bigg[ R_{x_0} \frac{\nom_{x_0} (\Fre_{x_0} \mathbb{E}[\nom])}{\mathbb{E}[\nom_{x_0}]^2} \bigg] \\
	&= \mean [ \Fre_{x_0} R ],
\intertext{and}
	\mean_\posterior{}[ \Fre_{x_0}^2 R ]
	&= \mean\bigg[ (\Fre_{x_0}^2R) \frac{\nom_{x_0}}{\mean[\nom_{x_0}]} \bigg]
	+ 2 \mean\bigg[ (\Fre_{x_0}R) \frac{\Fre_{x_0}\nom}{\mean[\nom_{x_0}]} \bigg] \\
	&\qquad - 2 \mean\bigg[ (\Fre_{x_0}R) \frac{\nom_{x_0} (\Fre_{x_0} \mean[\nom])}{\mean[\nom_{x_0}]^2} \bigg] \\
	& \qquad + \mean\bigg[ R_{x_0} \frac{\Fre_{x_0}^2\nom}{\mean[\nom_{x_0}]} \bigg] 
	- 2 \mean\bigg[ R_{x_0} \frac{(\Fre_{x_0}\nom) (\Fre_{x_0}\mean[\nom])}{\mean[\nom_{x_0}]^2} \bigg] \\  
    & \qquad - \mean\bigg[ R_{x_0} \frac{\nom_{x_0} (\Fre_{x_0}^2\mean[\nom])}{\mean[\nom_{x_0}]^2} \bigg] 
	+ 2 \mean\bigg[ R_{x_0} \frac{\nom_{x_0} (\Fre_{x_0}\mean[\nom])^2}{\mean[\nom_{x_0}]^3} \bigg] \\
	&= \mean[ \Fre_{x_0}^2R ]
	+ 2 \mean\bigg[ (\Fre_{x_0}R) \frac{\Fre_{x_0}\nom}{\nom_{x_0}} \bigg] 
	- 2 \mean\bigg[ (\Fre_{x_0}R) \frac{\Fre_{x_0} \mean[\nom]}{\nom_{x_0}} \bigg] \\
	&= \mean[ \Fre_{x_0}^2R ]
	+ 2 \mean[ (\Fre_{x_0}R) \langle \data-Q_{x_0}, \Fre_{x_0}Q \rangle_\Sigma ] \\
	&\qquad - 2 \mean[ \Fre_{x_0}R ] \langle \data-Q_{x_0}, \mean[ \Fre_{x_0}Q ] \rangle_\Sigma \\
	&= \mean[ \Fre_{x_0}^2R ]
	+ 2 \Cov[ \Fre_{x_0}R, \langle \data-Q_{x_0}, \Fre_{x_0}Q \rangle_\Sigma ] . 
\end{align*}
Thus by applying the mean to the Taylor expansion of $\mean_{\posterior{}}[R_{x_{\xi}}]$ we get \eqref{eq:approx3:exp_R}.
The approximation \eqref{eq:approx3:exp_RxR} is obtained by considering $\tilde{R}_\xi=R_{x_{\xi}}\otimes R_{x_{\xi}}$ and the approximation \eqref{eq:approx3:cov_R} is found when deducting the approximation of $\mean_\posterior{}[R_{x_{\xi}}]\otimes \mean_\posterior{}[R_{x_{\xi}}]$ from \eqref{eq:approx3:exp_RxR}.
\end{proof}
The expressions simplify further if the prior distribution is centered and skewless and we have some additional smoothness available.
\begin{corollary} \label{th:approximations4}
	Let the assumptions of \cref{th:approximations3} hold and assume that $\xi$ is centered, i.e., $\mean[\xi] = 0$. Then there holds
\begin{subequations} \label{eq:approx4}
\begin{align}
	\mean_\posterior{}[ R_{x_{\xi}} ]
	&= R_{x_0} \label{eq:approx4:exp_R} \\
	&\quad+ \frac{1}{2} \bigg( \mean[\Fre_{x_0}^2R] 
    + 2 \Cor[\Fre_{x_0}R , \langle\data-Q_{x_0}, \Fre_{x_0}Q\rangle_\Sigma] \bigg) \nonumber \\
    &\quad+ \LandauO(\|\xi\|_{L_{\prior{}}^3(\Omega;X)}^3) \nonumber \displaybreak[0] \\ 
	\Cor_\posterior{}[ R_{x_{\xi}} ]
	&= R_{x_0} \otimes R_{x_0} \label{eq:approx4:exp_RxR} \\
	&\quad + \frac{1}{2} \bigg( \mean[\Fre_{x_0}^2R] \otimes R_{x_0} 
    + R_{x_0} \otimes \mean[\Fre_{x_0}^2R] \nonumber \\ 
	&\qquad\quad + 2 \Cor[\Fre_{x_0}R ] \nonumber \\
	&\qquad\quad + 2 \Cor[(\Fre_{x_0}R \otimes R_{x_0} + R_{x_0} \otimes \Fre_{x_0}R) , \langle\data-Q_{x_0}, \Fre_{x_0}Q\rangle_\Sigma] \bigg) \nonumber \\  
    &\quad + \LandauO(\|\xi\|_{L_{\prior{}}^3(\Omega;X)}^3) \nonumber , \displaybreak[0] \\
	\Cov_\posterior{} [R_{x_{\xi}}] 
	&= \Cor[\Fre_{x_0}R] + \LandauO(\|\xi\|_{L_{\prior{}}^3(\Omega;X)}^3) \label{eq:approx4:cov_R} .
\end{align}
\end{subequations}
If $\xi \in L^4_{\prior{}}(\Omega;X)$ and $\xi$ is centered and additionally skewless, i.e., $\mean[\xi\otimes\xi\otimes\xi] = 0$, $Q \in C^4(X;\mathbb{R}^\nobs)$, and $R\in C^4(X;Z) \cap L^2_{\prior{\xi}}(X;Z)$, then the error term improves to $\LandauO(\|\xi\|_{L_{\prior{}}^4(\Omega;X)}^4)$.
\end{corollary}
\begin{proof}
To show the assertion \eqref{eq:approx4:exp_R} for the mean, continuing the derivations of \cref{th:approximations3}, we get
\begin{align*}
	\Fre_{x_0}^3\bigg( R \frac{\nom}{\mean[\nom]} \bigg)
	&= (\Fre_{x_0}^3R) \frac{\nom_{x_0}}{\mean[\nom_{x_0}]} 
	+ 3 (\Fre_{x_0}^2R) \frac{\Fre_{x_0}\nom}{\mean[\nom_{x_0}]}
	- 3 (\Fre_{x_0}^2R) \frac{\nom_{x_0} (\Fre_{x_0}\mean[\nom])}{\mean[\nom_{x_0}]^2} \\
	&\qquad + 3 (\Fre_{x_0}R) \frac{\Fre_{x_0}^2\nom}{\mean[\nom_{x_0}]} 
    - 6 (\Fre_{x_0}R) \frac{(\Fre_{x_0}\nom) (\Fre_{x_0}\mean[\nom])}{\mean[\nom_{x_0}]^2} \\
	&\qquad - 3 (\Fre_{x_0}R) \frac{\nom_{x_0} (\Fre_{x_0}^2\mean[\nom])}{\mean[\nom_{x_0}]^2}
	+ 6 (\Fre_{x_0}R) \frac{\nom_{x_0} (\Fre_{x_0}\mean[\nom])^2}{\mean[\nom_{x_0}]^3} \\
	&\qquad + R_{x_0} \frac{\Fre_{x_0}^3\nom}{\mean[\nom_{x_0}]} 
	- 3 R_{x_0} \frac{(\Fre_{x_0}^2\nom) (\Fre_{x_0}\mean[\nom])}{\mean[\nom_{x_0}]^2} \\
	&\qquad - 3 R_{x_0} \frac{(\Fre_{x_0}\nom) (\Fre_{x_0}^2\mean[\nom])}{\mean[\nom_{x_0}]^2} 
	+ 6 R_{x_0} \frac{(\Fre_{x_0}\nom) (\Fre_{x_0}\mean[\nom])^2}{\mean[\nom_{x_0}]^3} \\
	&\qquad - R_{x_0} \frac{\nom_{x_0} \Fre_{x_0}^3\mean[\nom]}{\mean[\nom_{x_0}]^2} 
	- 6 R_{x_0} \frac{\nom_{x_0} (\Fre_{x_0}\mean[\nom])^3}{\mean[\nom_{x_0}]^4} \\
	&\qquad + 6 R_{x_0} \frac{\nom_{x_0} (\Fre_{x_0}\mean[\nom]) (\Fre_{x_0}^2\mean[\nom])}{\mean[\nom_{x_0}]^3} .
\end{align*}
Again, since the Fr\'echet derivative acts as a linear operator in its direction of evaluation, $\mean[\xi]=0$ implies $\mean[\Fre_{x_0}R]=0$ and $\mean[\Fre_{x_0}Q]=0$, and thus also $\mean[\Fre_{x_0}\nom]=0$. Since the third Fr\'echet derivative acts as a trilinear operator in its directions of evaluation, the same arguments as in \cref{th:approximations3} yield $\mean[\Fre_{x_0}^3R]=0$ and $\mean[\Fre_{x_0}^3\nom]=0$.
Similarly, the multiplication of second derivatives (bilinear in direction of evaluation) and first derivatives (linear) is also trilinear in the direction of evaluation, so that we find
\begin{align*}
	\mean\bigg[ (\Fre_{x_0}^2R) \frac{\Fre_{x_0}\nom}{\mean[\nom_{x_0}]} \bigg] = 0 , &&
	\mean\bigg[ (\Fre_{x_0}R) \frac{\Fre_{x_0}^2\nom}{\mean[\nom_{x_0}]} \bigg] = 0 .
\end{align*}
Thus first- and third-order terms vanish under the assumptions of the corollary. The assertions for the correlation \eqref{eq:approx4:exp_RxR} and covariance \eqref{eq:approx4:cov_R} follow in complete analogy.
\end{proof}

\subsection{Comparison of the Perturbation Approximation to the Laplace Approximation}

\newcommand{\priord}{f}
\newcommand{\postd}{f^\delta}

In the following we compare our perturbation approach to the Laplace approximation, which is also based on a second-degree Taylor approximation.
To this end, let $X = \mathbb{R}^d$ and assume that the prior probability is given by a density $\priord(x): \mathbb{R}^d \to [0,\infty)$ with $S= \supp(\priord)$.
Let the posterior probability be unimodal and given by the posterior density
\begin{align*}
	\postd(x) 
	= \begin{cases}
	\frac{\exp(-\Phi(x)) \priord(x)}{\int_X \exp(-\Phi(x)) f(x) \dif x} = \frac{\exp(-I(x))}{\int_X \exp(-I(x)) \dif x} ,
	& x \in S , \\
	0 ,
	& \text{else} ,
	\end{cases}
\end{align*}
with $I:S \to \mathbb{R}$ the negative log-posterior density
\begin{align*}
	 I(x) = \Phi(x) - \log(\priord(x)).
\end{align*}
We assume $\Phi,f \in C^2(S)$, immediately implying also $I \in C^2(S)$.
Since the posterior density is unimodal the MAP point is given by
\begin{align*}
	x^\MAP = \argmax \limits_{x \in S} \postd(x) = \argmin \limits_{x \in S} I(x) .
\end{align*}
If we assume that $ x^{\MAP} $ is in the interior of $ S $, then $ \Fre_{x^{\MAP}} I[x] = 0 $. Moreover, since $ X $ is finite-dimensional, we can represent the second derivative $ \Fre_{x^{\MAP}}^2 I[x] $ as a Hessian matrix $ \nabla^2_{x^{\MAP}} I $. We assume that $ \nabla^2_{x^{\MAP}} I $ is positive definite.
The Laplace approximation is derived by replacing $I$ with its second-degree Taylor approximation, cf.~\cite{TK1986}, i.e.,
\begin{align*}
	\frac{\exp(-I(x))}{\int_X \exp(-I(x)) \dif x} 
	&\approx \frac{\exp(-I_{x^\MAP}-\tfrac{1}{2} x^\intercal (\nabla^2_{x^\MAP} I)x)}{\int_X \exp(-I_{x^\MAP}-\tfrac{1}{2} x^\intercal(\nabla^2_{x^\MAP} I) x) \dif x} 
	= \frac{\exp(-\tfrac{1}{2} x^\intercal(\nabla^2_{x^\MAP} I) x)}{\sqrt{(2\pi)^d \det(\nabla^2_{x^\MAP} I))}} .
\end{align*}
Thus, the Laplace approximation of $\postd$ is the Gaussian distribution \newline$\mathcal{N}(x^\MAP, (\nabla^2_{x^\MAP} I)^{-1})$.

Comparing the Laplace and the perturbation approximation of \cref{th:approximations3}, we can draw the following comparisons:
\begin{enumerate}
	\item The point of reference of the Laplace approximation's Taylor expansion is the (unique) MAP point, while the point of reference of the perturbation approximation can be any point in the support of the prior.
	\item The Laplace approximation works well when the posterior is close to a multivariate Gaussian distribution. In contrast, the accuracy of the perturbation approach can be arbitrarily increased by taking more derivatives into account.
	\item The Laplace approximation is a distribution, such that it requires additional sampling for the computation of statistical quantities of interest. On the other hand, the perturbation approach gives us a direct approximation of these quantities.
\end{enumerate}
Thus, while there are conceptual similarities, the perturbation approach from this article has a larger area for application.

\section{Iterative improvement of the reference value}\label{sec:impr}

\subsection{Iterative approximation of the posterior mean}\label{sec:itimpr}
The estimates of the previous section show that the reference point $x_0$ (and thus $R_{x_0}=R(x_0)$) for the series approximations is one of the most influential parameters for local sensitivity analysis for Bayesian inverse problems. For local sensitivity analysis for the forward problem, $x_0$ is often considered as an inherent property of the perturbation model \eqref{eq:domainmap}, which is specified by $x_0$ and the distribution of $\xi$. However, we may change $x_0$ by introducing a bias in the probability distribution of $\xi$, i.e., for any $\tilde{x}_0$ we may write
\begin{equation}\label{eq:introducebias}
	x_{\xi}(\omega)
	= x_0 + \xi(\omega)
	= \tilde{x}_0 + \big(\xi(\omega) + x_0-\tilde{x}_0\big)
	= \tilde{x}_0 + \tilde{\xi}(\omega) .
\end{equation}
We thus have changed the point of reference. The asymptotic expansions for the statistical quantities of interest from \cref{sec:perturbation} hold verbatim, with the moments in the expansions now taken with respect to the probability measure of $\tilde{\xi}$. The question is how we can exploit this flexibility in choosing the reference point to improve our asymptotic expansions. The starting point for our considerations is the following corollary.

\begin{corollary} \label{cor:approximationlinear}
Under the assumptions of \cref{th:approximations3} there holds
\begin{equation} \label{eq:approx3:exp_chi}
    \mean_\posterior{}[x_{\xi}]
    = x_0 
    + \mean[\xi] 
    + \Cov[\xi,\langle \data - Q_{x_0},\Fre_{x_0}Q[\xi]\rangle_\Sigma] 
    + \LandauO(\|\xi\|_{L_{\prior{}}^3(\Omega;X)}^3) .
\end{equation}
\end{corollary}
\begin{proof}
    Set $R_{x_{\xi}}=x_{\xi}$ in \eqref{eq:approx3:exp_R} from \cref{th:approximations3}.
\end{proof}

Being a new and hopefully improved approximation of $\mean_{\posterior{}}[x_{\xi}]$ when compared to $x_0$, we may take the result of this series approximation as a new point of reference for a series expansion of $\mean_{\posterior{}}[x_{\xi}]$. Iterating this process motivates the iteration
\begin{align}
	x^{(n+1)} &= \Phi_{\appr}(x^{(n)}) \label{eq:iteration:approx} \\
	&= x^{(n)}
	+ \mean[\xi^{(n)}] + \Cov[\xi^{(n)}, \langle\data-Q_{x^{(n)}}, \Fre_{x^{(n)}}Q[\xi^{(n)}]\rangle_\Sigma] , && n \in \mathbb{N}_0 , \nonumber
\end{align}
with starting value $x^{(0)}=x_0$ and the modified perturbation model
\[
x_{\xi}
=
x^{(n+1)}+\xi^{(n+1)},
\]
with
\begin{align}\label{eq:xiiteration}
\xi^{(0)}=\xi,
\qquad
\xi^{(n+1)}
=
\xi^{(n)}+x^{(n)}-x^{(n+1)}
=
\xi^{(0)}+x^{(0)}-x^{(n+1)}.
\end{align}
This yields the series expansion
\begin{align}
	\mean_{\posterior{}}[x_{\xi}] \label{eq:chin}
	= x^{(n)} + \mean[\xi^{(n)}] &+ \Cov [\xi^{(n)}, \langle\data-Q_{x^{(n)}}, \Fre_{x^{(n)}}Q[\xi^{(n)}] \rangle_\Sigma]\\
	&+ \LandauO\big(\max\{\|\xi\|_{L_{\prior{}}^3(\Omega;X)},\|x^{(0)}-x^{(n)}\|_X\}^3\big) ,
	&& n \in \mathbb{N}_0,\nonumber
\end{align}
with the constant in the $\LandauO$-notation depending on $\data,\Sigma, Q$, and $x_0$, cf.\ also \cref{lem:doubletaylor}.

If the sequence \eqref{eq:iteration:approx} converges, then the approximation estimate from \cref{cor:approximationlinear} becomes as follows.
\begin{corollary}\label{cor:improvedapproximation}
	Let the assumptions of \cref{th:approximations3} hold and let $x^\fixed$ be a fixed point of \eqref{eq:iteration:approx}. 
	Then there holds
	\[
	\mean_{\posterior{}}[x_{\xi}]
	=
	x^\fixed
	+
	\LandauO\big(\max\{\|\xi\|_{L_{\prior{}}^3(\Omega;X)},\|x^{(0)}-x^\fixed\|_X\}^3\big)
	\]
	with the constant in the $\LandauO$-notation depending on $\data,\Sigma, Q$, and $x_0$.
\end{corollary}

While we cannot guarantee existence and uniqueness of a fixed point for \cref{cor:improvedapproximation} in general, we can relate the iteration to classical inverse problems when dealing with Hilbert spaces.

\subsection{Connection to Tikhonov regularization}
In the following we show that the iteration \eqref{eq:iteration:approx} can also be motivated from a classical inverse problem perspective, if we allow $X$ to be a Hilbert space. In this situation there holds $\Cov[\xi]\in X\otimes X$, with $\otimes$ denoting the Hilbertian tensor product, and the covariance operator
\[
\mathcal{C}\colon X\to X,
\qquad
\mathcal{C} x=\big(\Id\otimes\langle\cdot,x\rangle_X\big)\Cov[\xi]
\]
is well defined. It is easy to check that $\mathcal{C}$ is continuous and self-adjoint w.r.t.\ $\langle\cdot,\cdot\rangle_X$. We moreover assume that $\Cov[\xi]$ is positive, implying that $\mathcal{C}$ is invertible.

With notation in place, we are in the position to introduce the Tikhonov-regularized classical inverse problem
\begin{equation}\label{eq:minprob}
\min_{x\in X}F(x),
\end{equation}
with $F\colon X\to\mathbb{R}$ given as
\[
F(x)
=
\frac{1}{2}\bigg(
\big\|\data - Q(x)\big\|_\Sigma^2
+
\big\|x-x^{(0)}-\mean[\xi]\big\|_\mathcal{C}^2
\bigg)
\]
and $\|\cdot\|_{\mathcal{C}}^2=\langle\mathcal{C}^{-1}\cdot,\cdot\rangle_X$.
Now, its Fr\'echet derivative $\Fre_xF\colon X\to\mathbb{R}$ at $x\in X$ evaluated in the direction $y\in X$ is given by
\[
(\Fre_xF)[y]
=
-\langle\data-Q_x,(\Fre_xQ)[y]\rangle_\Sigma
+
\big\langle\mathcal{C}^{-1}\big(x-x^{(0)}-\mean[\xi]\big),y\big\rangle_X,
\]
implying that its gradient $\nabla_xF \in X$ at $x\in X$ is given by
\[
\nabla_xF
=
-(\Fre_xQ)^*\Sigma^{-1}\big(\data-Q_x\big)
+
\mathcal{C}^{-1}\big(x-x^{(0)}-\mean[\xi]\big).
\]
Here, $(\Fre_xQ)^*$ is the adjoint taken with respect to the $\mathbb{R}^\nobs$ inner product. Now, upon observing that \eqref{eq:xiiteration} implies $\Cov[\xi^{(n)}]=\Cov[\xi^{(0)}]=\Cov[\xi]$ and $\mean[\xi^{(n)}]=\mean[\xi]+x^{(0)}-x^{(n)}$, the descent algorithm
\begin{equation}\label{eq:descentalgo}
x^{(n+1)}
=
x^{(n)}+d^{(n)}
\end{equation}
with
\begin{align*}
d^{(n)}
&=
-\mathcal{C}\nabla_{x^{(n)}}F\\
&=
\big(\Id\otimes\langle\cdot,(\Fre_{x^{(n)}}Q)^*\Sigma^{-1}\big(\data-Q_{x^{(n)}}\big)\rangle_X\big)\Cov[\xi]
+
\mean[\xi]+x^{(0)}-x^{(n)}\\
&=
\mean[\xi^{(n)}] + \Cov[\xi^{(n)}, \langle\data-Q_{x^{(n)}}, \Fre_{x^{(n)}}Q[\xi^{(n)}]\rangle_\Sigma]
\end{align*}
coincides with \eqref{eq:iteration:approx}. Our assumptions on $\Cov[\xi]$ guarantee that $\mathcal{C}$ is such that
\[
\langle d^{(n)},\nabla_{x^{(n)}}F\rangle_X
=
-\langle\mathcal{C}\nabla_{x^{(n)}}F,\nabla_{x^{(n)}}F\rangle_X
<
0,
\]
implying that $d^{(n)}$ is a descent direction for $F$. It is well known that, even with $d^{(n)}$ being a descent direction, convergence of the iteration can only be proven if $d^{(n)}$ is scaled with a sufficiently small step length $\alpha^{(n)}$. Still, if the iteration does indeed converge, i.e., $x^{(n)}\to x^\fixed$, then also $d^{(n)}\to 0$ and our assumptions on $\Cov[\xi]$ imply that $\nabla_{x^{(n)}}F\to 0$. Thus, if our iteration \eqref{eq:iteration:approx} converges, then its limit $x^\fixed$ is a (local) minimizer of \eqref{eq:minprob} whenever $(\Fre^2_{x^{\fixed}}F)[y]>0$ for all $y\in X$, $y\neq 0$. Vice versa, if the descent algorithm \eqref{eq:descentalgo} of the minimization problem \eqref{eq:minprob} converges, then it is a fixed point of our iteration \eqref{eq:iteration:approx} and we obtain the improved approximation estimate of \cref{cor:improvedapproximation}.

\section{Computational considerations} \label{sec:implementation}
\subsection{Affine-parametric representation of random variables}
For the purpose of numerical computations we assume in the following that the random variable $\xi$ can be parametrized over a sequence of real numbers. That is, we assume that
\begin{equation}\label{eq:affineparametric}
\xi(\omega)
=
\sum_{j=1}^M x_jz_j(\omega),
\end{equation}
with $M\in\mathbb{N}\cup\{\infty\}$ and where $\{x_j\}_{j=1}^M\subset X$ and the $z_j\colon\Omega\to\mathbb{R}$ are pairwise uncorrelated random variables. This implies the representation
\begin{equation}\label{eq:affineparametricx}
x_{\xi}(\omega)
=
x_0
+
\sum_{j=1}^M x_j z_j(\omega)
\end{equation}
for the input parameters in $X$. The most prominent example for such affine-parametric representations of $x_\xi$ are (truncated) Karhunen--Loève expansions. However, one can show that such expansions exist whenever $x_\xi$ attains values in a compact subset of a Banach space, see \cite[Section 1.2]{CD2015} for a discussion and examples. For numerical purposes, the expansion \eqref{eq:affineparametric} is truncated to a finite number of terms according to some tolerance.

\subsection{Computation of asymptotic expansions}
With the affine-parametric representation \eqref{eq:affineparametricx} available, the series representations from \eqref{eq:series} now read
\begin{subequations} \label{eq:KLE}

\begin{align}
    Q_{x_{\xi}}
    &= Q_{x_0} \label{eq:KLE:Q}
    + \sum \limits_{i=1}^M z_i\Fre_{x_0}Q[x_i] + \frac{1}{2} \sum \limits_{i=1}^M \sum \limits_{j=1}^M    z_iz_j\Fre_{x_0}^2Q[x_i,x_j] + \LandauO(\|\xi\|_X^3) ,  \\
    R_{x_{\xi}}
    &= R_{x_0} \label{eq:KLE:R}
    + \sum \limits_{i=1}^M z_i\Fre_{x_0}R[x_i] + \frac{1}{2} \sum \limits_{i=1}^M \sum \limits_{j=1}^M    z_iz_j\Fre_{x_0}^2R[x_i,x_j] + \LandauO(\|\xi\|_X^3) .
\end{align}
\end{subequations}
The directional derivatives occurring in \eqref{eq:KLE} can often be calculated conveniently if $Q(x), R(x)$ are explicitly available, as illustrated in the references in the literature review and in the numerical examples in \cref{sec:examples}. The following lemma shows that computing the entities in \eqref{eq:KLE} is sufficient to obtain approximations to the posterior moments.
\begin{corollary} \label{cor:KLE:uncentered}
	Let the assumptions of \cref{th:approximations3} hold and assume that the prior distribution of $x_\xi$ is given in the form as in \eqref{eq:affineparametricx} with $\{z_j\}_{j\in\mathbb{N}}$ being pairwise uncorrelated random variables satisfying $z_j\in L_{\prior{}}^2(\Omega)$. Then the expansions in \cref{th:approximations3} read
	\begin{subequations}
		\begin{align}
			\mean_\posterior{}[ R_{x_{\xi}} ]
			&= R_{x_0}
			+ \Fre_{x_0}R\bigg[\sum \limits_{i=1}^M \mean[z_i]x_i\bigg] \\
			&\qquad+ \frac{1}{2} \bigg( \sum \limits_{i=1}^M  \Cov[z_i] \Fre_{x_0}^2R[x_i]
			+ \Fre_{x_0}^2R\bigg[\sum \limits_{i=1}^M  \mean[z_i]x_i\bigg] \bigg) \nonumber \\
			&\qquad+\sum_{i=1}^M \Cov[z_i]\Fre_{x_0}R[x_i]\langle\data-Q_{x_0}, \Fre_{x_0}Q[x_i]\rangle_\Sigma 
			+\LandauO(\|\xi\|_{L_{\prior{}}^3(\Omega;X)}^3) \nonumber , \displaybreak[0] \\
			\Cor_\posterior{}[ R_{x_{\xi}} ]
			&= R_{x_0} \otimes R_{x_0} 
			+ \Fre_{x_0}R\bigg[\sum_{i=1}^M\mean[z_i]x_i\bigg]\otimes R_{x_0} + R_{x_0} \otimes \Fre_{x_0}R\bigg[\sum_{i=1}^M\mean[z_i]x_i\bigg]  \\
			&\qquad + \frac{1}{2}\bigg( 
			\sum \limits_{i=1}^M  \Cov[z_i] \Fre_{x_0}^2R[x_i]\otimes R_{x_0}
			+R_{x_0}\otimes \sum \limits_{i=1}^M  \Cov[z_i] \Fre_{x_0}^2R[x_i] \nonumber \\ 
			&\qquad\quad + \Fre_{x_0}^2R\bigg[\sum \limits_{i=1}^M  \mean[z_i]x_i\bigg]\otimes R_{x_0}
			+R_{x_0}\otimes\Fre_{x_0}^2R\bigg[\sum \limits_{i=1}^M  \mean[z_i]x_i\bigg]\bigg)\nonumber\\
			&\qquad+ \sum \limits_{i=1}^M \Cov[z_i] \Fre_{x_0}R[x_i]\otimes \Fre_{x_0}R[x_i] \nonumber\\			
			&\qquad+\Fre_{x_0}R\bigg[\sum \limits_{i=1}^M  \mean[z_i]x_i\bigg]\otimes \Fre_{x_0}R\bigg[\sum \limits_{i=1}^M  \mean[z_i]x_i\bigg] \nonumber \\
			& \qquad + \sum_{i=1}^M\Cov[z_i]\bigg(\Fre_{x_0}R[x_i]\otimes R_{x_0}+R_{x_0}\otimes\Fre_{x_0}R[x_i]\bigg)\nonumber\\
			&\qquad\qquad\qquad\qquad\qquad\qquad\qquad\cdot\langle\data-Q_{x_0}, \Fre_{x_0}Q[x_i]\rangle_\Sigma \nonumber \\ 
			&\qquad + \LandauO(\|\xi\|_{L_{\prior{}}^3(\Omega;X)}^3) \nonumber, \displaybreak[0] \\
			\Cov_\posterior{} [R_{x_{\xi}}] 
			&= \sum_{i=1}^M\Cov[z_i] \Fre_{x_0}R[x_i]\otimes \Fre_{x_0}R[x_i]
			+ \LandauO(\|\xi\|_{L_{\prior{}}^3(\Omega;X)}^3) .
		\end{align}
	\end{subequations}
	As for \cref{th:approximations4}, the first order correction terms vanish, if we assume that the $z_j$ are centered, i.e., $\mean[z_j]=0$. 
	Further, if $\xi \in L^4_{\prior{}}(\Omega;X)$ and $\xi$ is centered and additionally skewless, i.e., $\mean[\xi\otimes\xi\otimes\xi] = 0$, $Q \in C^4(X;\mathbb{R}^\nobs)$, and $R\in C^4(X;Z) \cap L^2_{\prior{\xi}}(X;Z)$, then the error term improves to $\LandauO(\|\xi\|_{L_{\prior{}}^4(\Omega;X)}^4)$.
\end{corollary}
\begin{proof}
Insert \eqref{eq:affineparametric} into \cref{th:approximations3} and \cref{th:approximations4}.
\end{proof}
\begin{remark}
    As for the deterministic case, this approach can be extended to arbitrary high order of the correction terms. However, when implemented in practice, the computational cost grows, if done naively, exponentially with the order of the derivatives. The efficient numerical solution of the arising \emph{moment equations} has been subject to intensive research, see, e.g., \cite{CS2013,HPS2013}, with an emphasis on being put on corrections up to second order. In the latter context, the moment equations are referred to as \emph{correlation equations} \cite{BG2004,DHP2017,DHS2017,HPS2012,HSS2008,vS2006}.
\end{remark}

\subsection{Iterative approximation of the posterior mean}
In the following we discuss how the iterative approximation of the posterior mean from \cref{sec:itimpr} can be computationally implemented when the random variable and input data are given in the affine-parametric representations \eqref{eq:affineparametric} and \eqref{eq:affineparametricx}. To this end, we assume that $X$ is a separable Hilbert space, which implies that we can assume without loss of generality that $\{x_j\}_{j\in\mathbb{N}}$ in \eqref{eq:affineparametric} is an orthonormal basis. Thus, 
\[
x_{\xi}(\omega)
= x_0 + \sum_{j=1}^M x_jz_j(\omega)
= \sum_{j=1}^M x_j\big(z_j(\omega)+\underbrace{\langle x_j,x_0\rangle_X}_{=:y_j}\big)
= \sum_{j=1}^M x_j\big(z_j(\omega)+y_j\big).
\]
Now, since for any $\alpha=[\alpha_1,\alpha_2,\ldots]$ and $\beta=[\beta_1,\beta_2,\ldots]$ there holds
\[
\bigg\langle\sum_{j=1}^M x_j\alpha_j,\sum_{j=1}^M x_j\beta_j\bigg\rangle_X
=
\langle\alpha,\beta\rangle_{\ell^2},
\]
the iteration from \cref{sec:itimpr} can be transferred to $\mathbb{R}^{M}$ equipped with the Euclidean $\ell^2$-inner product. To this end, set $y^{(0)}=[y_1,y_2,\ldots]$ and $z^{(0)}=[z_1,z_2,\ldots]$ and iterate
\[
y^{(n+1)}
=
y^{(n)}
+
\mean[z^{(n)}]
+
\Big[
\Cov[z_j^{(n)}]\langle \data-Q_{x^{(n)}},\Fre_{x^{(n)}}Q[x_j]\rangle_\Sigma
\Big]_{j=1}^M,
\]
where
\[
z^{(n+1)}
=
z^{(n)}+y^{(n)}-y^{(n+1)}
=
z^{(0)}+y^{(0)}-y^{(n+1)}
\]
and
\[
x^{(n)}
=
x_0 + \sum_{j=1}^M x_jz_j^{(n)}.
\]

\section{Numerical Examples} \label{sec:examples}

We consider two examples, a PDE, where the parameter space $X$ is a Hilbert space, and an ODE, where $X$ is a Banach space.

\subsection{Computational Setup}
All calculations were performed using MATLAB on a single node of the Marvin cluster of the University of Bonn. 
Each node has two Intel Xeon Platinum 8468 with forty-eight $2.1$ GHz cores, hyper-threading enabled and $1024$~GB~RAM. 
Calculations of the summed terms of \cref{cor:KLE:uncentered} as well as sampling were both done in parallel.

\subsection{Darcy flow with log-normal permeability}
As an example problem we consider the Darcy flow problem with permeability given by a log-normal random field, fixed load, and homogeneous boundary conditions, i.e.,
\begin{subequations}\label{eq:darcyflow}
\begin{align}
-\div(\exp(b(\omega))\nabla u(\omega))&=1&&\text{in}~\domain:=[0,1]^2,\\
u(\omega)&=0&&\text{on}~\partial \domain,
\end{align}
\end{subequations}
where $b(\omega)$ is a Gaussian random field taking values in $L^2(\domain)$ and $\omega$ is a random event in a probability space. This is an important situation in the geosciences \cite{BC2010,GN1999,TN1998}, where it is practically impossible to obtain complete knowledge on the permeability of the medium via direct measurements. Instead, practitioners are required to rely on measurements of the pressure $u$ at selected points in $\domain$. The situation corresponds to our framework when choosing $X=L^2(\domain)$, $Y=H_0^1(\domain)$, and $G$ as the forward response map $b(\cdot,\omega)\mapsto u(\cdot,\omega)$ given through \eqref{eq:darcyflow}. As observation operator $O\colon Y\to\mathbb{R}^\nobs$ we choose $\nobs=5$ point evaluations of $u$ at the locations
\[
\Big\{
\big(\tfrac{1}{2}, \tfrac{1}{2}\big),
\big(\tfrac{1}{4}, \tfrac{1}{4}\big),
\big(\tfrac{3}{4}, \tfrac{1}{4}\big),
\big(\tfrac{3}{4}, \tfrac{3}{4}\big),
\big(\tfrac{1}{4}, \tfrac{3}{4}\big)
\Big\}\subset \domain,
\]
see \cref{fig:sample} for an illustration.
We note that these point evaluations are well defined if $u$ is a pathwise sufficiently smooth function. To satisfy this requirement, we choose $b$ as a random field with mean $\mean[b]\equiv 1$ and covariance operator $\mathcal{C}\colon L^2(\domain)\to L^2(\domain)$ given through
\begin{align}\label{eq:covarianceoperator}
(\mathcal{C}\phi)(\vec{x})
=
\int_\domain\exp\bigg(-\frac{20\|\vec{x}-\vec{y}\|_2^2}{3}\bigg)\phi(\vec{y})\dd\vec{y},
\end{align}
and representation
\begin{equation}\label{eq:KL}
b(\vec{x},\omega)
=
b_0(\vec{x})
+
\underbrace{\alpha\sum_{i=1}^\infty b_i(\vec{x}) z_i(\omega)}_{=:\xi(\omega)},
\end{equation}
where $\alpha>0$, $b_0=\mean[b]$, $\{(b_i,\lambda_i)\}_{i=1}^\infty$ are the $L^2(\domain)$-orthonormal eigenpairs of $\mathcal{C}$ organized such that $\{\lambda_i\}_{i=1}^\infty$ is decreasing. We choose $z_i\sim\mathcal{U}[-\sqrt{\lambda_i},\sqrt{\lambda_i}]$ for experiments with centered random fields and $z_i\sim\mathcal{U}[-\sqrt{\lambda_i},\sqrt{\lambda_i}]+0.1$ for experiments with uncentered random fields. It is well known that this setting implies pathwise smoothness of $u$, such that $O$ is well defined. As prediction function we consider on the one hand the identity operator $R_1\colon X\to X=:Z_1$ and on the other hand the forward response map $R_2=G\colon X\to Y=:Z_2$. Finally, the $k$-th derivative of the forward response map $G$ is given as the solution of
\begin{subequations}\label{eq:darcyflowder}
\begin{align}
-\div\Big(\exp(b_0)\nabla \Fre_{b_0}^kG[\xi]\Big)
&=
\sum_{\ell=1}^k\binom{k}{\ell}\div\bigg(\exp(b_0)\xi(\omega)^\ell\nabla \Fre_{b_0}^{k-\ell}G[\xi]\bigg)&&\text{in}~\domain,\\
\Fre_{b_0}^kG[\xi]&=0&&\text{on}~\boundary,
\end{align}
\end{subequations}
see, e.g., \cite{BNT2007,BN2014,GN1999,TN1998}. Using the linearity of the observation operator and the definition of the prediction functions we obtain
\begin{align*}
\Fre_{b_0}^kQ[\xi]=O(\Fre_{b_0}^kG[\xi]), &&
\Fre_{b_0}^kR_2[\xi]=\Fre_{b_0}^kG[\xi], &&
\Fre_{b_0}^{1}R_1[\xi]=\xi, &&
\Fre_{b_0}^{1+k}R_1[\xi]=0,
\end{align*}
for $k \in \mathbb{N}$. An illustration of the log-normal random field, corresponding forward response, i.e., solution of the diffusion problem \eqref{eq:darcyflow}, its derivative and the evaluation points of the measurement operator can be found in \cref{fig:sample}.
\begin{figure}
	\centering
	\input{fig/sample/fig_u0.tex}
	\input{fig/sample/fig_randomfield.tex}\\
	\input{fig/sample/fig_u.tex}
	\input{fig/sample/fig_du.tex}
	\caption{Darcy model: 
	Top left: reference solution. 
	Top right: realization of random field. 
	Bottom left: perturbed solution according to random field top right. 
	Bottom right: Derivative in direction of random field top right. 
	Evaluation points highlighted.}	
	\label{fig:sample}
\end{figure}
We assume that the measurement noise has covariance
\[
[\Sigma]_{ij}
=
\frac{1}{1000}
\begin{cases}
5&i=j,\\
1&i\neq j,
\end{cases}
\quad
i,j=1,\ldots,5.
\]
We note that this noise level is quite small, and would be quite a challenge for most sampling based methods.

With the mathematical model set, we obtain measurement data $\eta^\delta$ through \eqref{eq:noisemodel}, i.e., we pick a random realization of $\xi(\omega)$, take it to be our ground truth, evaluate $Q(\xi(\omega))$, and add a realization of the measurement noise $\varepsilon\sim\mathcal{N}(0,\Sigma)$.

\subsubsection{Discretization}
We discretize the covariance operator \eqref{eq:covarianceoperator} and the Darcy flow problem \eqref{eq:darcyflow} and its derivatives \eqref{eq:darcyflowder} using continuous, piecewise linear finite elements on a triangular grid, leading to $N=481$ degrees of freedom for the Darcy flow problem and its derivatives and $545$ degrees of freedom for the perturbations. It is important to note that the limiting factor for the number of degrees of freedom is the computational cost of the Monte Carlo method (details below), which serves as a reference solution. The perturbation approach can easily deal with more degrees of freedom. As usual, to obtain an expansion of the form \eqref{eq:affineparametric}, we compute the eigendecomposition of the arising covariance matrix, and select the eigenfunctions corresponding to the largest eigenvalues with tolerance of $10^{-5}$, which yields an expansion of the form \eqref{eq:affineparametric} with $M=45$ terms. As a reference solution for our sensitivity analysis we use a QMC approximation of the variables of interest using $10^8$ Halton points.

\subsubsection{Local sensitivity analysis for Bayesian inverse problems}\label{sec:perturbationexperiments}
To confirm the theoretical convergence rates we consider in the following a scaled version of $\xi(\omega)$, i.e., we replace $\xi(\omega)$ in \eqref{eq:KL} by $\alpha\xi(\omega)$ with $\alpha\in\{2^{-n}\colon n=0,\ldots,15\}$ and compare the approximations obtained through local sensitivity analysis with the QMC reference solution for each $\alpha$. 
We calculate the approximations
\begin{align*}
\mean_\posterior{}[ R_{x_{\xi}} ]
&\approx \meanappr_\posterior{}[ R_{x_{\xi}} ]\\
&:= R_{x_0} 
+ \mean[\Fre_{x_0}R] \\ 
&\quad+ \frac{1}{2} \bigg( \mean[\Fre_{x_0}^2R] 
+ 2 \Cov[\Fre_{x_0}R, \langle\data-Q_{x_0}, \Fre_{x_0}Q\rangle_\Sigma ] \nonumber \bigg) \nonumber , \displaybreak[0] \\
\Cor_\posterior{}[ R_{x_{\xi}} ]
&\approx\operatorname{Cor}_\posterior{}[ R_{x_{\xi}} ]\\
&:= R_{x_0} \otimes R_{x_0} 
+ \bigg(\mean[\Fre_{x_0}R] \otimes R_{x_0} + R_{x_0} \otimes \mean[\Fre_{x_0}R]\bigg)\\
& \quad + \frac{1}{2} \bigg( \mean[\Fre_{x_0}^2R] \otimes R_{x_0} 
+ R_{x_0} \otimes \mean[\Fre_{x_0}^2R] 
+ 2 \Cor[\Fre_{x_0}R] \nonumber \\
& \qquad + 2 \Cov[\Fre_{x_0}R \otimes R_{x_0} + R_{x_0} \otimes \Fre_{x_0}R,\langle\data-Q_{x_0}, \Fre_{x_0}Q\rangle_\Sigma] \bigg) \nonumber, \displaybreak[0] \\
\Cov_\posterior{} [R_{x_{\xi}}] 
&\approx
\operatorname{Cov}_\posterior{} [R_{x_{\xi}}] 
:= \Cov[\Fre_{x_0}R],
\end{align*}
for the case of uncentered prior distributions and whose approximation properties are given by of \cref{th:approximations3}. 
For centered prior distributions we compute the simplified approximations for mean and correlation according to
\begin{align*}
\mean_\posterior{}[ R_{x_{\xi}} ]
&\approx
\meanappr_\posterior{}[ R_{x_{\xi}} ]
=
 R_{x_0}
+ \frac{1}{2} \bigg( \mean[\Fre_{x_0}^2R] 
+ 2 \Cor[\Fre_{x_0}R , \langle\data-Q_{x_0}, \Fre_{x_0}Q\rangle_\Sigma] \bigg) , \displaybreak[0] \\
\Cor_\posterior{}[ R_{x_{\xi}} ]
&\approx
\operatorname{Cor}_\posterior{}[ R_{x_{\xi}} ]\\
&= R_{x_0} \otimes R_{x_0} \\
&\quad + \frac{1}{2} \bigg( \mean[\Fre_{x_0}^2R] \otimes R_{x_0} 
+ R_{x_0} \otimes \mean[\Fre_{x_0}^2R] + 2 \Cor[\Fre_{x_0}R ] \nonumber \\
&\qquad\quad + 2 \Cor[(\Fre_{x_0}R \otimes R_{x_0} + R_{x_0} \otimes \Fre_{x_0}R) , \langle\data-Q_{x_0}, \Fre_{x_0}Q\rangle_\Sigma] \bigg) \nonumber,
\end{align*}
with approximation properties given by \cref{th:approximations4}.
As error measure we use the $\|\cdot\|_{L^2(\domain)}$-norm, for the mean and the $\|\cdot\|_{L^2(\domain) \otimes L^2(\domain)}$-norm, for the covariance and the correlation. 
We note that computing the approximation using local sensitivity analysis for all $\alpha$ requires less than a second in total, most of which was spent building the system matrices, whereas the QMC reference solution requires 33 minutes for each $\alpha$. 
Using \cref{cor:KLE:uncentered}, the memory requirements for the calculation of the approximations are dictated by the required memory for computation of the first and second order derivatives of $R$. 
Thus, the memory required scales linearly in $M$, the number of KLE-terms, and the degrees of freedom of the discretization $N$.
The calculation of these derivatives amounts to the solution of the same linear PDE as the computation of the reference point and thus requires the same amount of memory. Thus, the overall memory requirement is $\mathcal{O}(MN)$.

The plots of the error are shown in \cref{fig:conv_exp_cov3} for the uncentered case and in \cref{fig:conv_exp_cov4} for the centered and skewless case compared to the predicted rates.
The predicted rates are confirmed, albeit for some small perturbations, the error of the reference solution is still too large to see the actual error of the perturbation approximation.
\begin{figure}
	\centering
	\definecolor{mycolor1}{rgb}{0.00000,0.44700,0.74100}%
\definecolor{mycolor2}{rgb}{0.85000,0.32500,0.09800}%
\begin{tikzpicture}

\begin{axis}[%
width=.3\textwidth,
height=.4\textwidth,
at={(0,0)},
scale only axis,
xmode=log,
xmin=1e-05,
xmax=1,
xminorticks=true,
xlabel style={font=\color{white!15!black}},
xlabel={$\alpha\sim\|\xi\|_{L_{\prior{}}^3(\Omega;X)}$},
ymode=log,
ymin=2.22044604925031e-16,
ymax=100,
yminorticks=true,
ylabel style={font=\color{white!15!black}},
ylabel={$\|\cdot\|_{L^2(\domain)}$-error},
axis background/.style={fill=white},
title style={font=\bfseries},
title={Error $\mean_\posterior{}$},
axis x line*=bottom,
axis y line*=left,
legend style={at={(-0,-0.25)}, anchor=north west, legend cell align=left, align=left, draw=white!15!black},
legend columns=1,
]

\addplot [color=mycolor2, dashed, mark=+, mark options={solid, mycolor2}]
table[row sep=crcr]{%
	3.05175781249999e-05	2.56638620705176e-09\\
	6.10351562499999e-05	5.13264231181433e-09\\
	0.0001220703125	1.02642966848216e-08\\
	0.000244140625	2.05208897378055e-08\\
	0.000488281249999999	4.09809134546828e-08\\
	0.000976562499999999	8.14815496733774e-08\\
	0.001953125	1.59363326962674e-07\\
	0.00390625	3.0145487973897e-07\\
	0.00781249999999999	9.2622948767306e-07\\
	0.015625	8.60724207858656e-06\\
	0.03125	9.14253114953558e-05\\
	0.0625	0.00101295283382424\\
	0.125	0.0108123975298309\\
	0.25	0.0932363036485665\\
	0.5	0.53479941407133\\
	1	2.31353548507061\\
};
\addlegendentry{$\meanappr_\posterior{}[ R_{1,x_{\xi}} ]$}
\addplot [color=mycolor1, dashed, mark=+, mark options={solid, mycolor1}]
  table[row sep=crcr]{%
3.05175781249999e-05	3.81235906307102e-11\\
6.10351562499999e-05	7.62397800483087e-11\\
0.0001220703125	1.52417928552616e-10\\
0.000244140625	3.04331536103302e-10\\
0.000488281249999999	6.04559312938978e-10\\
0.000976562499999999	1.17576868394131e-09\\
0.001953125	2.08347045140297e-09\\
0.00390625	2.55077074537105e-09\\
0.00781249999999999	2.42052033535986e-08\\
0.015625	2.8898612332007e-07\\
0.03125	3.15453800687475e-06\\
0.0625	3.565387055653e-05\\
0.125	0.000387583524573277\\
0.25	0.00339152449724368\\
0.5	0.019339268238941\\
1	0.0827015462921294\\
};
\addlegendentry{$\meanappr_\posterior{}[ R_{2,x_{\xi}} ]$}

\addplot [color=black, dashed, mark options={solid, black}]
  table[row sep=crcr]{%
3.05175781249999e-05	2.8421709430404e-14\\
6.10351562499999e-05	2.27373675443232e-13\\
0.0001220703125	1.81898940354585e-12\\
0.000244140625	1.45519152283668e-11\\
0.000488281249999999	1.16415321826935e-10\\
0.000976562499999999	9.31322574615479e-10\\
0.001953125	7.45058059692383e-09\\
0.00390625	5.96046447753907e-08\\
0.00781249999999999	4.76837158203124e-07\\
0.015625	3.81469726562499e-06\\
0.03125	3.05175781249999e-05\\
0.0625	0.000244140625\\
0.125	0.001953125\\
0.25	0.015625\\
0.5	0.125\\
1	1\\
};
\addlegendentry{$\LandauO(\alpha^3)$}

\end{axis}
\end{tikzpicture}%
	\definecolor{mycolor1}{rgb}{0.00000,0.44700,0.74100}%
\definecolor{mycolor2}{rgb}{0.85000,0.32500,0.09800}%
\begin{tikzpicture}

\begin{axis}[%
width=.3\textwidth,
height=.4\textwidth,
at={(0,0)},
scale only axis,
xmode=log,
xmin=1e-05,
xmax=1,
xminorticks=true,
xlabel style={font=\color{white!15!black}},
xlabel={$\alpha\sim\|\xi\|_{L_{\prior{}}^3(\Omega;X)}$},
ymode=log,
ymin=2.22044604925031e-16,
ymax=1,
yminorticks=true,
ylabel style={font=\color{white!15!black}},
ylabel={$\|\cdot\|_{L^2(\domain)\otimes L^2(\domain)}$-error},
axis background/.style={fill=white},
title style={font=\bfseries},
title={Error $\Cov_\posterior{}$, $\Cor_\posterior{}$},
axis x line*=bottom,
axis y line*=left,
legend style={at={(-0.2,-0.25)}, anchor=north west, legend cell align=left, align=left, draw=white!15!black},
legend columns=2,
]

\addplot [color=mycolor2, dashed, mark=+, mark options={solid, mycolor2}]
table[row sep=crcr]{%
	3.05175781249999e-05	4.43176158699274e-10\\
	6.10351562499999e-05	8.37556081983343e-10\\
	0.0001220703125	1.48541728357793e-09\\
	0.000244140625	2.28356259555212e-09\\
	0.000488281249999999	3.35609899779404e-09\\
	0.000976562499999999	1.69678321082085e-08\\
	0.001953125	9.25950595007098e-08\\
	0.00390625	4.40626469985022e-07\\
	0.00781249999999999	2.03737657196801e-06\\
	0.015625	9.93095560383724e-06\\
	0.03125	5.57985623423683e-05\\
	0.0625	0.000386330612997737\\
	0.125	0.00302577471516177\\
	0.25	0.0200487455305973\\
	0.5	0.0924493765235618\\
	1	0.374328695254915\\
};
\addlegendentry{$\operatorname{Cor}_\posterior{}[ R_{1,x_{\xi}} ]$}

\addplot [color=mycolor2, dashed, mark=x, mark options={solid, mycolor2}]
table[row sep=crcr]{%
	3.05175781249999e-05	8.5454969486371e-15\\
	6.10351562499999e-05	3.41759009212429e-14\\
	0.0001220703125	1.36667683328611e-13\\
	0.000244140625	5.46398448336349e-13\\
	0.000488281249999999	2.19147897742987e-12\\
	0.000976562499999999	9.4139017714151e-12\\
	0.001953125	6.87128154125066e-11\\
	0.00390625	9.67101982380812e-10\\
	0.00781249999999999	1.52882552047056e-08\\
	0.015625	2.41152513955182e-07\\
	0.03125	3.69969807623037e-06\\
	0.0625	5.19860245874488e-05\\
	0.125	0.000574203247906418\\
	0.25	0.00439106305332262\\
	0.5	0.0271606137305571\\
	1	0.13045109578408\\
};
\addlegendentry{$\operatorname{Cov}_\posterior{}[ R_{1,x_{\xi}} ]$}
\addplot [color=mycolor1, dashed, mark=+, mark options={solid, mycolor1}]
  table[row sep=crcr]{%
3.05175781249999e-05	2.73510195628616e-12\\
6.10351562499999e-05	5.46809628892715e-12\\
0.0001220703125	1.09254743152425e-11\\
0.000244140625	2.17895660834983e-11\\
0.000488281249999999	4.31825293660176e-11\\
0.000976562499999999	8.35286580186913e-11\\
0.001953125	1.45215008731532e-10\\
0.00390625	1.31621151672991e-10\\
0.00781249999999999	1.62354471296157e-09\\
0.015625	2.05525320296912e-08\\
0.03125	2.29406952273669e-07\\
0.0625	2.64088429736852e-06\\
0.125	2.92146239808345e-05\\
0.25	0.000261304529580716\\
0.5	0.00151243863859356\\
1	0.00647766683446887\\
};
\addlegendentry{$\operatorname{Cor}_\posterior{}[ R_{2,x_{\xi}} ]$}

\addplot [color=mycolor1, dashed, mark=x, mark options={solid, mycolor1}]
  table[row sep=crcr]{%
0.0001220703125	7.48913448563991e-17\\
0.000244140625	4.11367257081568e-16\\
0.000488281249999999	2.5284263039097e-15\\
0.000976562499999999	1.58677063201395e-14\\
0.001953125	8.79672824049045e-14\\
0.00390625	4.34430614463487e-13\\
0.00781249999999999	1.03960571638955e-11\\
0.015625	2.24874940028269e-10\\
0.03125	3.94863046706345e-09\\
0.0625	5.90107450674705e-08\\
0.125	6.69815289179994e-07\\
0.25	5.21214111422185e-06\\
0.5	3.20195022779435e-05\\
1	0.000144287277690394\\
};
\addlegendentry{$\operatorname{Cov}_\posterior{}[ R_{2,x_{\xi}} ]$}

\addplot [color=black, dashed, mark options={solid, black}]
  table[row sep=crcr]{%
3.05175781249999e-05	2.8421709430404e-14\\
6.10351562499999e-05	2.27373675443232e-13\\
0.0001220703125	1.81898940354585e-12\\
0.000244140625	1.45519152283668e-11\\
0.000488281249999999	1.16415321826935e-10\\
0.000976562499999999	9.31322574615479e-10\\
0.001953125	7.45058059692383e-09\\
0.00390625	5.96046447753907e-08\\
0.00781249999999999	4.76837158203124e-07\\
0.015625	3.81469726562499e-06\\
0.03125	3.05175781249999e-05\\
0.0625	0.000244140625\\
0.125	0.001953125\\
0.25	0.015625\\
0.5	0.125\\
1	1\\
};
\addlegendentry{$\LandauO(\alpha^3)$}

\end{axis}
\end{tikzpicture}%
	\caption{
		Convergence for Darcy model and uncentered prior distributions. 
	}
	\label{fig:conv_exp_cov3}
\end{figure}

\begin{figure}
	\centering
	\definecolor{mycolor1}{rgb}{0.00000,0.44700,0.74100}%
\definecolor{mycolor2}{rgb}{0.85000,0.32500,0.09800}%
\begin{tikzpicture}

\begin{axis}[%
width=.3\textwidth,
height=.4\textwidth,
at={(0,0)},
scale only axis,
xmode=log,
xmin=1e-05,
xmax=1,
xminorticks=true,
xlabel style={font=\color{white!15!black}},
xlabel={$\alpha\sim\|\xi\|_{L_{\prior{}}^4(\Omega;X)}$},
ymode=log,
ymin=2.22044604925031e-16,
ymax=1,
yminorticks=true,
ylabel style={font=\color{white!15!black}},
ylabel={$\|\cdot\|_{L^2(\domain)}$-error},
axis background/.style={fill=white},
title style={font=\bfseries},
title={Error $\mean_\posterior{}$},
axis x line*=bottom,
axis y line*=left,
legend style={at={(-0,-0.25)}, anchor=north west, legend cell align=left, align=left, draw=white!15!black},
legend columns=1,
]

\addplot [color=mycolor2, dashed, mark=+, mark options={solid, mycolor2}]
  table[row sep=crcr]{%
3.05175781249999e-05	4.01857578540162e-12\\
6.10351562499999e-05	8.03711558759622e-12\\
0.0001220703125	1.60736892223361e-11\\
0.000244140625	3.21388613830744e-11\\
0.000488281249999999	6.41442177790902e-11\\
0.000976562499999999	1.26441218090763e-10\\
0.001953125	2.62147314942606e-10\\
0.00390625	2.34522872776988e-09\\
0.00781249999999999	3.85493715329702e-08\\
0.015625	6.18956907377809e-07\\
0.03125	9.82191443726266e-06\\
0.0625	0.000151622676488386\\
0.125	0.00210577561038483\\
0.25	0.0202290919197898\\
0.5	0.109936986295113\\
1	0.497308279829146\\
};
\addlegendentry{$\meanappr_\posterior{}[ R_{1,x_{\xi}} ]$}

\addplot [color=mycolor1, dashed, mark=+, mark options={solid, mycolor1}]
  table[row sep=crcr]{%
3.05175781249999e-05	4.54086619763742e-14\\
6.10351562499999e-05	9.0818007363888e-14\\
0.0001220703125	1.81714662794535e-13\\
0.000244140625	3.62311520292108e-13\\
0.000488281249999999	7.09793675854375e-13\\
0.000976562499999999	1.20562187229507e-12\\
0.001953125	4.74894634848997e-12\\
0.00390625	9.75164069175805e-11\\
0.00781249999999999	1.61576846288884e-09\\
0.015625	2.59084867824908e-08\\
0.03125	4.110163299955e-07\\
0.0625	6.34271722422275e-06\\
0.125	8.79512417040639e-05\\
0.25	0.000837256301936738\\
0.5	0.00429785957914631\\
1	0.0178543076357566\\
};
\addlegendentry{$\meanappr_\posterior{}[ R_{2,x_{\xi}} ]$}

\addplot [color=black, dashed, mark options={solid, black}]
  table[row sep=crcr]{%
6.10351562499999e-05	1.38777878078145e-17\\
0.0001220703125	2.22044604925032e-16\\
0.000244140625	3.55271367880051e-15\\
0.000488281249999999	5.68434188608077e-14\\
0.000976562499999999	9.09494701772923e-13\\
0.001953125	1.45519152283668e-11\\
0.00390625	2.32830643653869e-10\\
0.00781249999999999	3.7252902984619e-09\\
0.015625	5.96046447753904e-08\\
0.03125	9.53674316406247e-07\\
0.0625	1.52587890625e-05\\
0.125	0.000244140625\\
0.25	0.00390625\\
0.5	0.0625\\
1	1\\
};
\addlegendentry{$\LandauO(\alpha^4)$}

\end{axis}
\end{tikzpicture}%
	\definecolor{mycolor1}{rgb}{0.00000,0.44700,0.74100}%
\definecolor{mycolor2}{rgb}{0.85000,0.32500,0.09800}%
\begin{tikzpicture}

\begin{axis}[%
width=.3\textwidth,
height=.4\textwidth,
at={(0,0)},
scale only axis,
xmode=log,
xmin=1e-05,
xmax=1,
xminorticks=true,
xlabel style={font=\color{white!15!black}},
xlabel={$\alpha\sim\|\xi\|_{L_{\prior{}}^4(\Omega;X)}$},
ymode=log,
ymin=2.22044604925031e-16,
ymax=1,
yminorticks=true,
ylabel style={font=\color{white!15!black}},
ylabel={$\|\cdot\|_{L^2(\domain)\otimes L^2(\domain)}$-error},
axis background/.style={fill=white},
title style={font=\bfseries},
title={Error $\Cov_\posterior{}$, $\Cor_\posterior{}$},
axis x line*=bottom,
axis y line*=left,
legend style={at={(-0.2,-0.25)}, anchor=north west, legend cell align=left, align=left, draw=white!15!black},
legend columns=2,
]

\addplot [color=mycolor2, dashed, mark=+, mark options={solid, mycolor2}]
  table[row sep=crcr]{%
6.10351562499999e-05	3.83411659643945e-17\\
0.0001220703125	3.46261365305541e-16\\
0.000244140625	5.17109994226683e-15\\
0.000488281249999999	8.24136154187585e-14\\
0.000976562499999999	1.31853521905392e-12\\
0.001953125	2.10972873227315e-11\\
0.00390625	3.3754608248142e-10\\
0.00781249999999999	5.39972265208515e-09\\
0.015625	8.63287655922704e-08\\
0.03125	1.37694680815867e-06\\
0.0625	2.17479160778098e-05\\
0.125	0.000328358840327718\\
0.25	0.00394829054639808\\
0.5	0.02667306596041\\
1	0.129789671235362\\
};
\addlegendentry{$\operatorname{Cor}_\posterior{}[ R_{1,x_{\xi}} ]$}

\addplot [color=mycolor2, dashed, mark=x, mark options={solid, mycolor2}]
  table[row sep=crcr]{%
6.10351562499999e-05	3.99049228064369e-17\\
0.0001220703125	3.79841805385784e-16\\
0.000244140625	5.70407803634797e-15\\
0.000488281249999999	9.08134612052697e-14\\
0.000976562499999999	1.45229286619331e-12\\
0.001953125	2.32343790207599e-11\\
0.00390625	3.71709025318657e-10\\
0.00781249999999999	5.94507614901085e-09\\
0.015625	9.49777363854143e-08\\
0.03125	1.5105493262474e-06\\
0.0625	2.36066206486353e-05\\
0.125	0.000345676929825376\\
0.25	0.00401012158244311\\
0.5	0.0270034843837886\\
1	0.130439476671303\\
};
\addlegendentry{$\operatorname{Cov}_\posterior{}[ R_{1,x_{\xi}} ]$}
\addplot [color=mycolor1, dashed, mark=+, mark options={solid, mycolor1}]
  table[row sep=crcr]{%
3.05175781249999e-05	3.23911533723845e-15\\
6.10351562499999e-05	6.47860495816937e-15\\
0.0001220703125	1.29585909725432e-14\\
0.000244140625	2.58127621764823e-14\\
0.000488281249999999	5.01118235403453e-14\\
0.000976562499999999	7.79840932939721e-14\\
0.001953125	3.99238778528443e-13\\
0.00390625	8.48401460977598e-12\\
0.00781249999999999	1.40379127351973e-10\\
0.015625	2.25066124480234e-09\\
0.03125	3.57170357690754e-08\\
0.0625	5.51973243023162e-07\\
0.125	7.69443198094908e-06\\
0.25	7.42750935567882e-05\\
0.5	0.000383996322858865\\
1	0.001582091649879\\
};
\addlegendentry{$\operatorname{Cor}_\posterior{}[ R_{2,x_{\xi}} ]$}

\addplot [color=mycolor1, dashed, mark=x, mark options={solid, mycolor1}]
  table[row sep=crcr]{%
0.000488281249999999	1.39036371487736e-16\\
0.000976562499999999	2.22711434516285e-15\\
0.001953125	3.55948499813848e-14\\
0.00390625	5.69465964311051e-13\\
0.00781249999999999	9.10775438303587e-12\\
0.015625	1.45492154841591e-10\\
0.03125	2.31321343434586e-09\\
0.0625	3.61037359549685e-08\\
0.125	5.25639915769298e-07\\
0.25	5.9300385477247e-06\\
0.5	3.6854171078436e-05\\
1	0.000163721813630701\\
};
\addlegendentry{$\operatorname{Cov}_\posterior{}[ R_{2,x_{\xi}} ]$}

\addplot [color=black, dashed, mark options={solid, black}]
  table[row sep=crcr]{%
6.10351562499999e-05	1.38777878078145e-17\\
0.0001220703125	2.22044604925032e-16\\
0.000244140625	3.55271367880051e-15\\
0.000488281249999999	5.68434188608077e-14\\
0.000976562499999999	9.09494701772923e-13\\
0.001953125	1.45519152283668e-11\\
0.00390625	2.32830643653869e-10\\
0.00781249999999999	3.7252902984619e-09\\
0.015625	5.96046447753904e-08\\
0.03125	9.53674316406247e-07\\
0.0625	1.52587890625e-05\\
0.125	0.000244140625\\
0.25	0.00390625\\
0.5	0.0625\\
1	1\\
};
\addlegendentry{$\LandauO(\alpha^4)$}

\end{axis}
\end{tikzpicture}%
	\caption{
	Convergence for Darcy model and centered prior distributions. 
	}
	\label{fig:conv_exp_cov4}
\end{figure}
\subsubsection{Iterative improvement of the reference point}
In the following we aim to show that the iteration from \cref{sec:impr}, if convergent, can indeed improve the reference value of the Taylor expansion and that the approximation rate from \cref{cor:improvedapproximation} indeed holds. 
To this end, as in \cref{sec:perturbationexperiments}, we consider a scaled version of $\xi(\omega)$ and run the iteration \eqref{eq:iteration:approx} for each scaled version of $\xi(\omega)$. 
The magnitude of the updates to the reference values is illustrated in \cref{fig:iterationupdates}. The result after 100 iterations is then compared to the QMC reference solution for each $\alpha$, the error being measured in the $\|\cdot\|_{L^2(\domain)}$-norm and plotted in \cref{fig:conv_iteration}.
We can see that the iteration diverges for the largest two perturbation values since the default step size is too large. 
The other iterations are converging with the larger perturbations showing significantly faster rates of convergence.
The smallest calculated perturbations are still very far from their fixed point relative to the expected error. 
For intermediate sizes of perturbations the rate predicted in \cref{cor:improvedapproximation} can be confirmed.
\begin{figure}
	\centering
	\input{fig/iteration/fig_iteration.tex}
	\caption{$\|\cdot\|_{L^2(\domain)}$-norm of updates $d^{(n)}$ for each $\alpha$ for $100$ iterations for Darcy model.}
	\label{fig:iterationupdates}
\end{figure}

\begin{figure}
	\centering
	\definecolor{mycolor1}{rgb}{0.00000,0.44700,0.74100}%
\definecolor{mycolor2}{rgb}{0.85000,0.32500,0.09800}%
\begin{tikzpicture}

\begin{axis}[%
width=.66\textwidth,
height=.4\textwidth,
at={(0,00)},
scale only axis,
xmode=log,
xmin=1e-03,
xmax=1,
xminorticks=true,
xlabel style={font=\color{white!15!black}},
xlabel={$\alpha\sim\|\xi\|_{L_{\prior{}}^3(\Omega;X)}$},
ymode=log,
ymin=1e-07,
ymax=2.6499043776189,
ylabel style={font=\color{white!15!black}},
ylabel={$\|\cdot\|_{L^2(\domain)}$-error},
yminorticks=true,
axis background/.style={fill=white},
title style={font=\bfseries},
title={Error at iteratively improved reference point},
axis x line*=bottom,
axis y line*=left,
legend style={at={(.66\textwidth,.15)}, anchor=south east, legend cell align=left, align=left, draw=white!15!black}
]
\addplot [color=mycolor2, dashed, mark=+, mark options={solid, mycolor2}]
  table[row sep=crcr]{%
0.00781249999999999	2.63142639706165e-05\\
0.015625	2.8598712638154e-05\\
0.03125	1.5430815142116e-05\\
0.0625	0.000133411393157643\\
0.125	0.00119493677234557\\
0.25	0.0141552088329052\\
0.5	2.64990437761889\\
};
\addlegendentry{$\mean_{\posterior{}}[x_{\xi}]-x^{(100)}$}

\addplot [color=black, dashed, mark options={solid, black}]
  table[row sep=crcr]{%
0.00781249999999999	4.76837158203124e-07\\
0.015625	3.81469726562499e-06\\
0.03125	3.05175781249999e-05\\
0.0625	0.000244140625\\
0.125	0.001953125\\
0.25	0.015625\\
0.5	0.125\\
1	1\\
};
\addlegendentry{$\LandauO(\alpha^3)$}

\end{axis}
\end{tikzpicture}%
	\caption{Error of $x^{(100)}$ compared to QMC approximation of $\mean_\posterior{}[R_2]=\mean_\posterior{}[x_\xi]$ in $\|\cdot\|_{L^2(\domain)}$-norm for Darcy model.
	}
	\label{fig:conv_iteration}
\end{figure}

\subsection{Lotka--Volterra model with Brownian Bridge perturbation}

As a second example we consider a Lotka--Volterra predator-prey model where the perturbation $\xi(t)$ is a scaled Brownian bridge, i.e.,
\begin{align*}
	\frac{\dd}{\dd t} y_1(t,\omega)
	&= 
	\bigg(\frac{15}{2} + \xi(t,\omega)\bigg) y_1(t,\omega) 
	- \frac{3}{40} y_1(t,\omega) y_2(t,\omega) 
	& 
	t & \in [0,1] , \\
	\frac{\dd}{\dd t} y_2(t,\omega) 
	&= 
	\frac{3}{20} y_1(t,\omega) y_2(t,\omega) 
	- \frac{15}{2} y_2(t,\omega) 
	& 
	t & \in [0,1] ,
\intertext{with initial values $y_1(0) = y_2(0) = 20$ and}
	\xi(t,\omega)
	&= \alpha \sum \limits_{i=1}^{M}  \frac{\sqrt{2}}{k\pi} \sin(k\pi t)~ z_i(\omega) 
	&
	t &\in [0,1].
\end{align*}
with $z_i(\omega) \sim \mathcal{N}(0,1)$. 
Note that $\xi$ is continuous, but not differentiable, i.e., $X = C^0([0,1])$. 
Clearly $X$ is a Banach space when equipped with the maximum norm.
We observe $y \in Y = [C^1([0,1])]^2$ at $t\in\{\tfrac{1}{4},\tfrac{1}{2},\tfrac{3}{4},1\}$, i.e., $Y=[\mathbb{R}^2]^4\simeq\mathbb{R}^8$.
For the covariance matrix of the observation noise, we choose
\begin{align*}
	\Sigma &= \sigma  \Sigma_0 , & 
	\Sigma_0 &= 
	I_4 \otimes	
	\begin{pmatrix}
	1 & \tfrac{1}{10} \\
	\tfrac{1}{10} & 1
	\end{pmatrix}
\end{align*}
with $I_4 \in \mathbb{R}^{4\times 4}$ the identity matrix and $\sigma$ is a scaling parameter.
That is, observations at different times are uncorrelated and observations of predator and prey at the same point in time are slightly correlated.
We focus on the case $Z=X$ and estimate $\mean_{\posterior{}}[\xi(t)]$.

To calculate the approximations, we need $\Fre_{x_0} Q = \Fre_{x_0} y$ which is given by the ODE
\begin{align*}
	\frac{\dd}{\dd t} \bigg(\Fre_{x_0}y_1(t)[\xi]\bigg)
	&= 
	\frac{15}{2} \Fre_{x_0} y_1(t)[\xi] 
	- \frac{3}{40} \bigg( \Fre_{x_0}y_1(t)[\xi] y_2(t) 
	+ y_1(t) \Fre_{x_0}y_2(t)[\xi] \bigg) , \\
	\frac{\dd}{\dd t} \bigg(\Fre_{x_0}y_2(t)[\xi]\bigg)
	&= 
	\frac{3}{20} \bigg(\Fre_{x_0}y_1(t)[\xi] y_2(t) 
	+ y_1(t) \Fre_{x_0}y_2(t)[\xi] \bigg) 
	- \frac{15}{2} \Fre_{x_0}y_2(t)[\xi] ,
\end{align*}
$t \in [0,1]$, with initial values $\Fre_{x_0}y_1(0)[\xi] = \Fre_{x_0}y_2(0)[\xi] = 0$.

In \cref{fig:ode_sample} the unperturbed trajectory of the ODE is illustrated with a sample of the perturbed solution and its first-order Taylor approximation. 
The underlying sample of the Brownian bridge is depicted in the second picture.
\begin{figure}
	\centering
	\input{fig/LV/sample/fig_ode.tex}
	\input{fig/LV/sample/fig_BB.tex}
	\caption{Lotka--Volterra model:
	Left: unperturbed trajectory, perturbed trajectory and first-order Taylor approximation of perturbed trajectory. 
	Right: sample of Brownian bridge.} \label{fig:ode_sample}
\end{figure}
In the following, we use the measurement data $\data = [ (97,19) , (46,333) , (7,86) , (20,20)]\in  [\mathbb{R}^2]^4$.

\subsubsection{Discretization}

For the evaluation of the {ODE}s we use $1000$~equidistant time steps and a predictor-corrector algorithm based on the explicit and implicit Euler scheme. 
To this end, we performs one step with the explicit Euler method which is then taken as the starting point for five steps of the fix point iteration of the implicit Euler method. 
We truncate the Brownian bridge after $M = 100$ terms.

\subsubsection{Local sensitivity analysis for Bayesian inverse problems}

We again check the convergence rate of the centered approximations of the posterior mean for $\alpha \in \{2^{-n}: n=0,\ldots,7\}$.
As a reference estimate to calculate our error we use an antithetic Monte Carlo approximation of $10^7$ samples.
To show the influence of the observation noise on the approximation, we calculate the error for $\sigma \in \{5,10,20\}$.
The maximum norm of the estimation error are shown in~\cref{fig:LV_conv4}.
\begin{figure}
	\centering
	\definecolor{mycolor1}{rgb}{0.06600,0.44300,0.74500}%
\definecolor{mycolor2}{rgb}{0.86600,0.32900,0.00000}%
\definecolor{mycolor3}{rgb}{0.92900,0.69400,0.12500}%
\begin{tikzpicture}

\begin{axis}[%
width=.66\textwidth,
height=.4\textwidth,
at={(0,0)},
scale only axis,
xmode=log,
xmin=1e-3,
xmax=1,
xminorticks=true,
xlabel={$\alpha \sim \|\xi\|_{L_{\prior{}}^4(\Omega;X)}$},
ymode=log,
ymin=1e-10,
ymax=100,
yminorticks=true,
title style={font=\bfseries},
title={Error $\mean_\posterior{}$},
ylabel={$\|\cdot\|_{C^0([0,1])}$-norm},
axis background/.style={fill=white},
axis x line*=bottom,
axis y line*=left,
legend style={at={(1,0.03)}, anchor=south east, legend cell align=left, align=left}
]
\addplot [color=mycolor1, dashed, mark=star, mark options={solid, mycolor1}]
  table[row sep=crcr]{%
0.00781249999999999	6.50940855129479e-06\\
0.015625	9.55277885613269e-05\\
0.03125	0.00136285878204135\\
0.0625	0.0158437123223724\\
0.125	0.123678886292646\\
0.25	0.661393169489952\\
0.5	2.92080615460051\\
1	12.0308164363915\\
};
\addlegendentry{$\sigma = 5$}

\addplot [color=mycolor2, dashed, mark=star, mark options={solid, mycolor2}]
  table[row sep=crcr]{%
0.00781249999999999	1.73915425863037e-06\\
0.015625	2.44223397510451e-05\\
0.03125	0.000360002430209483\\
0.0625	0.00477330217421364\\
0.125	0.0464473474070584\\
0.25	0.295522059319028\\
0.5	1.40824310291517\\
1	5.95371226509448\\
};
\addlegendentry{$\sigma = 10$}

\addplot [color=mycolor3, dashed, mark=star, mark options={solid, mycolor3}]
  table[row sep=crcr]{%
0.00781249999999999	4.88509112273625e-07\\
0.015625	6.21976915125915e-06\\
0.03125	9.1152161144133e-05\\
0.0625	0.00130558342808157\\
0.125	0.0153561884152966\\
0.25	0.121816312459019\\
0.5	0.658227318931618\\
1	2.91701368203035\\
};
\addlegendentry{$\sigma = 20$}

\addplot [color=black, dashed, mark options={solid, black}]
  table[row sep=crcr]{%
0.00781249999999999	3.72529029846191e-09\\
0.015625	5.96046447753907e-08\\
0.03125	9.53674316406247e-07\\
0.0625	1.52587890625e-05\\
0.125	0.000244140625\\
0.25	0.00390625\\
0.5	0.0625\\
1	1\\
};
\addlegendentry{$\LandauO(\alpha^4)$}

\end{axis}
\end{tikzpicture}%
	\caption{Convergence for Lotka--Volterra model and centered prior distribution for different noise levels.} \label{fig:LV_conv4}
\end{figure}
The calculation of the approximations took less than a second and the Monte Carlo approximations took 27~minutes for each $\alpha$. 
The memory required again scales linearly with the number of KLE-terms $M$ and the degrees of freedom of the discretization, i.e., the number of time steps.

\subsubsection{Iterative improvement of the reference point}

We also test the iteration on the Lotka--Volterra model.
In \cref{fig:LV_iteration} the magnitude of the updates is shows for the first 100 iterations.
\begin{figure}
	\centering
	\input{fig/LV/fig_lv_iteration.tex}
	\caption{$\|\cdot\|_{C^0([0,1])}$-norm of updates $d^{(n)}$ for 100 iterations.} \label{fig:LV_iteration}
\end{figure}
The result after 100 iterations is then compared against an antithetic Monte Carlo approximation and the respective error shown in~\cref{fig:LV_iteration_conv}.
\begin{figure}
	\centering
	\definecolor{mycolor1}{rgb}{0.00000,0.44700,0.74100}%
\definecolor{mycolor2}{rgb}{0.85000,0.32500,0.09800}%
\begin{tikzpicture}

\begin{axis}[%
width=.66\textwidth,
height=.4\textwidth,
at={(0,0)},
scale only axis,
xmode=log,
xmin=1e-3,
xmax=1,
xminorticks=true,
xlabel={$\alpha \sim \|\xi\|_{L_{\prior{}}^3(\Omega;X)}$},
ymode=log,
ymin=1e-07,
ymax=1,
yminorticks=true,
ylabel={$\|\cdot\|_{C^0([0,1])}$-error},
axis background/.style={fill=white},
title style={font=\bfseries},
title={Error at iteratively improved referene point},
axis x line*=bottom,
axis y line*=left,
legend style={at={(0.97,0.03)}, anchor=south east, legend cell align=left, align=left}
]
\addplot [color=mycolor2, dashed, mark=+, mark options={solid, mycolor2}]
  table[row sep=crcr]{%
0.00781249999999999	6.68711349055228e-05\\
0.015625	0.00084734338718433\\
0.03125	0.00904290673661537\\
0.0625	0.049909841341203\\
0.125	0.149970643807502\\
0.25	0.179349593463464\\
};
\addlegendentry{$\mean_{\posterior{}}[x_\xi] - x^{(100)}$}

\addplot [color=black, dashed, mark options={solid, black}]
  table[row sep=crcr]{%
0.00781249999999999	4.76837158203125e-07\\
0.015625	3.814697265625e-06\\
0.03125	3.0517578125e-05\\
0.0625	0.000244140625\\
0.125	0.001953125\\
0.25	0.015625\\
0.5	0.125\\
1	1\\
};
\addlegendentry{$\LandauO(t^3)$}

\end{axis}
\end{tikzpicture}%
	\caption{Error of $x^{(100)}$ compared to MC approximation of $\mean_{\posterior{}}[x_\xi]$ in $\|\cdot\|_{C^0([0,1])}$-norm.} \label{fig:LV_iteration_conv}
\end{figure}

\section{Conclusion} \label{sec:conclusion}
In this article we combined asymptotic expansions from local sensitivity analysis with the framework of Bayesian inverse problems, i.e., we derived Taylor-like expansions for first and second moments  of predictions functions with respect to the posterior measure. 
To mitigate the influence of the a-priorily chosen reference point, we derived an iterative scheme to improve this reference point for the case when the prediction function is the identity and discussed its connection to classical inverse problems and Tikhonov regularization. 
The analysis was carried out in infinite dimensional spaces, and is thus applicable to a large class of problems, particularly to partial differential equations. 
We provided such an example in our numerical experiments, where our numerical examples for the Darcy flow with log-normal permeability and the Lotka--Volterra ODE with Brownian bridge perturbation confirmed our analytical approximation rates.

\bibliographystyle{siamplain}
\bibliography{Bibliography.bib}
\end{document}